\documentclass[11pt,a4paper]{article}
\usepackage[utf8]{inputenc}

\usepackage{amssymb}
\usepackage{amsmath}
\usepackage{amsthm}
\usepackage{amsfonts}
\usepackage[left=1in,bottom=1in,top=1in,right=1in]{geometry}
\usepackage[colorlinks=true]{hyperref}

\usepackage[disable]{todonotes}
\presetkeys{todonotes}{inline}{}

\usepackage{csquotes}
\usepackage[boxruled]{algorithm2e}

\usepackage{thmtools,thm-restate}

\makeatletter
\if@todonotes@disabled

\else

\fi
\makeatother

\usepackage{commath}

\theoremstyle{definition}
\newtheorem{definition}{Definition}[section]

\theoremstyle{remark}

\theoremstyle{plain}
\newtheorem{theorem}{Theorem}[section]

\theoremstyle{plain}
\newtheorem{corollary}{Corollary}[theorem]

\theoremstyle{plain}
\newtheorem{lemma}[theorem]{Lemma}

\theoremstyle{plain}

\newtheorem{fact}[theorem]{Fact}

\newcommand{\br}{\mathbb{R}}
\newcommand{\brn}[1][n]{\mathbb{R}^#1}

\newcommand{\ip}[2]{\left\langle #1, #2 \right\rangle}

\title{\textbf{Sampling and Optimization on Convex Sets in Riemannian Manifolds of Non-Negative Curvature}}
\author{Navin Goyal \thanks{Microsoft Research India; Email: navingo@microsoft.com} \and Abhishek Shetty \thanks{Microsoft Research India; Email: ashetty1995@gmail.com}}

\begin{document}
	\maketitle
	\begin{abstract}
		The Euclidean space notion of convex sets (and functions) generalizes to Riemannian manifolds in a natural sense and is called geodesic convexity.
		Extensively studied computational problems such as convex optimization and sampling in convex sets also have meaningful counterparts
		in the manifold setting. Geodesically convex optimization is a well-studied problem with ongoing research and considerable recent interest in machine learning
		and theoretical computer science. In this paper, we study sampling and convex optimization problems over manifolds of non-negative
		curvature proving polynomial running time in the dimension and other relevant parameters. Our algorithms assume a warm start. 
		We first present a random walk based sampling algorithm and then combine it with simulated annealing for solving convex optimization problems. 
		To our knowledge, these are the first algorithms in the general setting of positively curved manifolds with provable polynomial guarantees under reasonable assumptions, and the first study of the connection between sampling and optimization in this setting. 
	\end{abstract}
	
	
	
		\section{Introduction}
	
    \subsection{Motivation}
    
    Convex geometry in the Euclidean space is a well-developed area of mathematics with connections to many fields. 
    Purely from a geometric perspective, convex sets form a rich class of well-behaved sets with many fascinating properties. 
    Convex optimization, which is a mainstay in the theory of optimization with far reaching applications, provides further motivation for the extensive study of convexity.
    The notion of convex relaxation has led to advances in our understanding of approximation algorithms.
    In algorithmic convex geometry, one is interested in constructing algorithms to answer natural questions about convex sets and functions. 
    One important question in the area is sampling uniformly from convex sets and more generally from log-concave distributions. 
    The problem of designing efficient sampling algorithms has led to connections to probability and geometry, raising interesting questions about isoperimetric inequalities and the analysis of Markov operators; see e.g. \cite{vempala2005geometric} and \cite{vempala2010recent} for a survey of the area. 
    
    In the Euclidean space, there is an intimate connection between sampling, volume computation (more generally integration) and optimization.  
    Though polynomial time deterministic volume computation for convex bodies is known to have exponential approximation ratio in the worst case, using efficient sampling, one can construct a fully polynomial time randomized approximation scheme for volume computation \cite{Dyer:1991:RPA:102782.102783}; a long line of follow up work has further improved the time complexity. 
    Similarly, the ability to sample efficiently can be used to design efficient optimization algorithms \cite{bertsimas2004solving, kalai2006simulated}. 
     
    A natural generalization of convexity is the notion of \emph{geodesic convexity} on Riemannian manifolds (see e.g., \cite{MR1326607}). In addition to being motivated from a mathematical perspective, this notion enjoys applications in theoretical computer science and machine learning. 
    Even in convex geometry in the Euclidean space, giving a Riemannian structure to the convex set has been instructive in understanding the behavior of algorithms such as the interior point methods (e.g. see \cite{MR1930942, lee2017geodesic, lee2018convergence}) and is in a sense, a natural way to approach the problem. 
    That said, the theory of algorithmic convex geometry on Riemannian manifolds doesn't seem as well developed as its Euclidean counterpart and it is a natural direction to explore both for its intrinsic interest as well as for applications.     
    In this paper, we ask the following question and take some modest steps towards answering it:
    
     \begin{center} \emph{Is there a theory of algorithmic convex geometry on manifolds? In particular, to what extent do the relations between sampling and optimization that hold in the Euclidean case carry over to manifolds?} \end{center}
    
    In the Euclidean case, most sampling algorithms are based on geometric random walks.
    Examples of these walks include the ball walk, where in each step one samples uniformly from the ball at the current point, rejecting the step if it lands outside the convex set and the Hit-and-Run walk, which in each step samples uniformly from a random chord through the current point. 
    These algorithms are generally analyzed by relating their conductance to isoperimetric properties of the underlying set. 
    A technique from convex geometry that has been instrumental in Euclidean convex geometry is that of localization. 
    Localization is a technique used to prove dimension free inequalities by reducing the high dimensional problem to a one-dimensional problem by decomposing the space.  
    The technique has been used to prove isoperimetric and related inequalities for convex bodies and more generally for log-concave measures and has been useful for the analysis of sampling algorithms. 
    Until recently, the existence of a Riemannian analogue of the technique was unclear due to the lack of symmetries in the manifold setting. \cite{MR3709716} showed that localization technique does indeed extend to the setting of Riemannian manifolds with Ricci curvature lower bounds by exploiting connections to the theory of optimal transport. 
    The connections between optimal transport on manifolds and isoperimetric inequalities had been considered earlier in the context of proving Pr\'{e}kopa--Leindler type inequalities on manifolds with Ricci curvature lower bounds \cite{MR2295207}. 
    See \hyperref[sec:Needle]{Section \ref*{sec:Needle}} for an extended discussion.

%
%
%

    A class of algorithms that relates sampling and optimization in the Euclidean space is the so called cutting plane methods. 
    As an example, consider the center of gravity method which using sampling estimates the center of gravity and then uses a gradient oracle to restrict to a halfspace that contains the minimum. 
    Using the Gr\"{u}nbaum inequality, one can see that each of the two parts of the convex set has a constant fraction of the volume of the convex body and thus each iteration reduces the volume of the set of interest by a constant fraction. It is natural to ask if
    this technique can be extended to manifolds. 
     \cite{rusciano2018riemannian} proves a version of the Gr\"{u}nbaum inequality in the case of Hadamard manifolds, showing that there exist points such that any halfspaces through these points have large fraction of the volume of the set.
    Using this they show an oracle complexity result for optimization, where each oracle call returns a gradient at a center point. 
   It remains unclear whether these methods can be made efficient, due in part to the fact that it is unclear how the required center points can be found efficiently.
    
    Another family of efficient convex optimization algorithms in the Euclidean space is the interior point methods. 
    These algorithms construct barrier functions for the convex sets of interest and use these functions to reduce a constrained optimization problem to an unconstrained optimization problem. 
    As noted earlier, this can be seen as optimizing on a Riemannian manifold with the Hessian structure induced by the barrier function. 
    This connection has been fruitful in constructing optimization algorithms and has also provided motivation for considering optimization on Riemannian manifolds. 
    For optimization on convex sets in Riemannian manifold, one could ask if interior point methods have a natural analog. 
    To our knowledge, there doesn't exist any general construction of this nature. 
    Recently, \cite{Abernethy:2016:FCO:3045390.3045656} show that the interior point methods and simulated annealing techniques are  related in the Euclidean setting: for the universal barrier on the convex body the central path and the heat path of the simulated annealing algorithm coincide suggesting a randomized algorithm for using the barrier using sampling. 
    This connection further motivates studying simulated annealing on manifolds with the hope that this will shed further light on interior point methods in the manifold setting. 
    
    On Riemannian manifolds, there exists a natural notion of uniform sampling analogous to the notion of the Lebesgue measure in the Euclidean space. 
    This is given by the Riemannian volume form (\hyperref[def:vol_form]{Definition \ref*{def:vol_form}}). 
    With this in mind, it makes sense to ask for random walks that have the uniform distribution as the stationary distribution. 
    It is instructive to note that, in general, the natural ball walk on a compact manifold does not have as its stationary distribution the uniform distribution on the manifolds. 
    By the natural ball walk, we mean the walk where at each step we sample uniformly (i.e. from the measure induced on the ball from the natural Riemannian volume form) from the (geodesic) ball of a fixed radius and move to the sampled point. 
    The issue is that, on manifolds, the volume of balls varies with the center and the stationary distribution of the walk is proportional to this volume. 
    In order to get the right distribution as the stationary point, one can use a Metropolis filter but that requires the knowledge of the volume at each point in order to execute the walk. 
    In the setting of manifolds, it is even unclear how to sample uniformly from the ball at each point as required by the ball walk.
    
    In this paper, we consider manifolds with positive sectional curvature. 
    Note that the positivity of the sectional curvature is a stronger constraint than the positivity of the Ricci curvature, since the former implies the latter. 
    Examples of these include the sphere in $ n $ dimensions, the set of orthogonal matrices and Lie groups with bi-invariant metrics. 
    Manifolds of positive curvature have been extensively studied in differential geometry. 
    Their geometry and topology is well understood with characterization theorems such as the Cheeger--Gromoll splitting theorem, Bonnet--Myers theorem and the soul theorem. 
	Their geometry can also be understood using comparison theorems which allow us to relate geometric quantities (such as Jacobi fields) about these manifolds to the corresponding quantities in the Euclidean space (which has zero curvature).
	The notion of convexity on manifolds has also received attention. For example, \cite{MR0339008} proves the non-existence of convex functions on compact manifolds, which requires us to restrict attention to convex subsets of manifolds of our interest. 
    Optimization on these groups has been considered in literature for various applications such as independent component analysis; for example, see \cite{10.1007/978-3-540-30110-3_157}.
   
    
    Sampling and optimization on geodesically convex sets has several applications. 
    For example, recently the geodesically convex optimization on the positive semidefinite matrices has been applied to capacity and operator scaling problems. 
    Closely related is the application of geodesically convex optimization to the Brascamp--Lieb constants \cite{sra_et_al:LIPIcs:2018:9429}. 
    There are many applications of geodesically convex optimization in machine learning and other areas; for example see \cite{zhang2016first} for a discussion of the applications.
    We note however that many of these applications involve manifolds with nonpositive curvature and hence do not fall under our setting.
     
    \subsection{Previous Work}
    
    Literature on sampling algorithms is rich and diverse.
    Developments in the field have led to several deep results in probability and computer science. 
    For sampling from convex sets, many walks such as the grid walk, the ball walk and the Hit-and-Run walk have been analyzed. 
    \cite{Dyer:1991:RPA:102782.102783} provide an approximation scheme for the volume of a convex body using sampling, which further motivated the study of sampling on a convex body. 
    \cite{lovasz1990mixing} and \cite{lovasz1993random} study mixing time of walks such as the ball walk using localization and isoperimetry and these results were further improved by \cite{kannan1995isoperimetric}. 
    \cite{lovasz1999hit} and \cite{lovasz2006hit} analyze the hit-and-run walk and show that it has better mixing properties compared to the ball walk. 
    \cite{lovasz2007geometry} analyze the sampling from log-concave distributions using the ball and the hit-and-run walk. 
    \cite{lee2017geodesic}, \cite{lee2018convergence}, \cite{mangoubi2017rapid}, \cite{NIPS2018_7842} and \cite{lee2018algorithmic} provide further improvements on sampling uniformly from convex sets and from logconcave densities. 
    
    There has been recent interest in sampling on manifolds. 
    \cite{MR3586941} surveys methods for sampling on manifolds with emphasis on statistical applications. 
    \cite{MR2599200} analyzes the spectrum of the kernel of the natural ball walk and of the Metropolis adjusted walk on manifolds. 
    They show that the convergence of both the walks to the stationary distributions in terms of the eigenvalues of the Laplace--Beltrami operator on the manifold. 
    As noted earlier, in the setting of manifolds, the natural ball walk does not converge to the uniform distribution. 
    
    \cite{mangoubi2016rapid} shows that the geodesic walk on manifolds with positive sectional curvature mixes rapidly. Indeed, they show that the walk mixes in time independent of the dimension of the manifold. 
    The metric they consider for comparison with the stationary distribution is the Wasserstein distance, which is the natural choice for their analysis.
    In their analysis, they use the fact that on manifolds of positive curvature, geodesics that are initially ``parallel" tend to move towards each other, a fact captured by the Rauch comparison theorem. 
    As an example of this consider the sphere and two lines of longitude starting at the equator. 
    As we move up towards the north pole, the distance between the end point of the arc reduces till the lines finally intersect at the north pole. 
    This contraction property can easily be converted to statement about average distances, which implies mixing in the Wasserstein metric.
    They use the sampling algorithm thus obtained to sample from the surface of a convex body using a walk whose mixing time does not depend on the ambient dimension. 
    
    \cite{lee2017geodesic} consider the problem of sampling from a Riemannian manifold whose metric is specified by the Hessian of a convex function, with the aim to provide faster sampling times for sampling uniformly from a polytope. 
    The motivation for considering the Hessian structure on a convex set is motivated from the theory of interior point methods where the Riemannian structure induced by the barrier function acts as a natural lens through which to view the optimization technique.
    \cite{lee2017geodesic} show that a walk similar to the geodesic walk that we consider mixes rapidly on Hessian manifolds.
    They use this to obtain a faster algorithm for sampling from a polytope. 
    They state their isoperimetry theorems for Hessian manifolds stated in terms of parameters which are bounded for the particular choice of convex function.    
    \todo[disable]{Note about Hessian manifolds and requirement of positivity, compare to ours}
    
    Another perspective on the geodesic walk is as Hamiltonian dynamics on the manifolds with constant potential. 
    This leads to connection to Hamiltonian Monte Carlo, which is a well-studied technique in Markov chain Monte Carlo first considered in \cite{duane1987hybrid}. 
    \todo[disable]{Rephrase?}
    The Hamiltonian dynamics was original considered in physics as an alternate description of classical mechanics.
    The idea of Hamiltonian Monte Carlo is to sample paths that are solutions of the Hamiltonian equations.
    In Hamiltonian Monte Carlo, the Hamiltonian dynamics is used to produce a Markov chain whose stationary distribution is the required distribution. 
    For a survey of the area, see \cite{MR2858447} and \cite{betancourt2017conceptual}. 
    The Riemannian version of the Hamiltonian Monte Carlo was first considered in \cite{MR2814492}. 
    In a followup work to \cite{lee2017geodesic}, \cite{lee2018convergence} show convergence guarantees for the walk towards the stationary distribution which is set to be a Gibbs distribution on a manifold.
    Here too they show, isoperimetric theorems for Hessian manifolds specified by convex functions which has a convex second derivative. 
    They then apply to adapt the Gaussian cooling algorithm on the Hessian manifold induced by the logarithmic barrier to compute volume on polytopes. 
    
    Though optimization of geodesically convex functions on manifolds is a problem that has been considered for a long time, it remains an active area of research, e.g. see \cite{MR1326607}, \cite{MR1480415} and \cite{MR2364186}. It has received attention recently in data science, machine learning and theoretical computer science. 
    For example, recent work on operator scaling and the Brascamp Lieb inequality considers geodesically convex optimization on the manifold of positive semidefinite matrices; e.g. \cite{allen2018operator}, \cite{sra_et_al:LIPIcs:2018:9429} and \cite{vishnoi2018geodesic}. 
    Optimization techniques based on adapting descent algorithms such as gradient descent, accelerated gradient descent  and Newton's method have been analyzed in the manifold setting.
    For an overview of descent based algorithms, see \cite{zhang2016first}, \cite{zhang2018towards}, \cite{zhang2018cubic} and references therein.
    Though analyses of descent algorithms on manifolds have been considered for some time, for example see \cite{MR1326607}, most results concerned asymptotic convergence. 
    \cite{zhang2016first} consider unconstrained optimization on Hadamard manifolds, which are complete, simply connected manifolds with non-positive sectional curvature, and show bounds on the convergence rates of various first order algorithms such as gradient descent and stochastic gradient descent. 
    \cite{allen2018operator} consider a particular geodesically convex optimization problem on the positive semidefinite cone (seen as a manifold) arising from operator scaling and show exponential rate of convergence of a box-constrained Newton's method using specific properties of the objective function. 
    \cite{zhang2018towards} analyze accelerated methods on convex sets on manifolds but assume that successive iterations remain within the convex set.
    \cite{zhang2018r} analyze an accelerated stochastic gradient based algorithm on Riemannian manifolds.  
    One advantage that these gradient based methods enjoy is that their rates are independent of the dimension. But when doing constrained
    optimization on geodesically convex sets they require ``projection'' to the convex set whose complexity is not clear.
    To our knowledge, techniques for optimization based on sampling have not been analyzed in the setting of Riemannian manifolds. 
    
    \subsection{Our Contributions}
    In this paper, we consider the problem of sampling on geodesically convex sets on manifolds. 
    We provide an algorithm that samples from the Riemannian volume form on geodesically convex sets on manifolds with non-negative lower bounds on the curvature. This corresponds to uniform sampling. More generally, our algorithm can sample from log-concave distributions on geodesically convex sets. 
    We show how to use the sampling algorithm to optimize over geodesically convex sets on manifolds, adapting the simulated annealing algorithm to this setting. 
   
   	Our sampling algorithm has access to the manifold via an oracle for the exponential map and to the convex set via a membership oracle.
   	 We remark that the manifold need not be given extrinsically i.e. by an embedding in an ambient Euclidean space, and our algorithm can work with intrinsic view of manifolds. 
    For sampling on the whole manifold, \cite{mangoubi2016rapid} show results that are similar in spirit to our results. In fact, they show that when we are interested in the whole manifold, the geodesic walk mixes in time independent of the dimension, but with the additional requirement that the (sectional) curvature must be lower bounded away from zero.
    Our results use different techniques, i.e. conductance, to bound the mixing time and require just non-negativity of the curvature.
    When the whole manifold is considered, our techniques give mixing assuming only bounds on the Ricci curvature. 
    Another point to be noted is that \cite{mangoubi2016rapid} show mixing with respect to the Wasserstein metric which on compact metric spaces is weaker than the the total variation metric. 
    But, using standard coupling techniques and bounding the total variation distance between transition kernels of nearby points (as in \autoref{thm:one_step}), bounds on the total variation mixing can be achieved. 
    
    The sampling results of \cite{lee2017geodesic} and \cite{lee2018convergence} are for Hessian manifolds obtained from barrier functions on convex sets in the Euclidean space, while we work on convex subsets of manifolds with positive curvature.
   	They too use conductance to bound the mixing time. 
   	Localization is used to show their conductance results but their results are for Hessian manifolds with conditions on the convex function whose Hessian generates the metric. 
    
    
    To our knowledge, the current paper is the first to explore the connection between sampling and optimization on manifolds. 
    Our algorithm is a zeroth order optimization algorithm requiring access only to an evaluation oracle to the function. 
    \todo[disable]{rephrase next two sentences, spell out Lipschitz and regularity}
    Due to this fact, we only assume that the function is convex and has bounded Lipschitz constant (which can be achieved by assuming bounded first order derivatives) while first order and second order methods or zeroth order methods that estimate the gradient require bounds on higher derivatives. 
    We assume regularity conditions on the convex set by assuming that the convex set contains a ball of radius $ r $ and has small diameter. 
    To our knowledge, no previous algorithm was known for geodesically-convex optimization in a general settings such as ours. The algorithm of \cite{allen2018operator} works
    on the PSD manifold with specific objective function.
    
   
   	We make extensive use of the localization techniques from \cite{MR3709716}.
   	Firstly, we show that the isoperimetric theorem from \cite{kannan1995isoperimetric} has a natural analogue for manifolds with non-negative curvature and use the localization to show facts about the isoperimetric profile of convex sets in these manifolds.
   	Secondly, we use localization to bound the expectation of the function value when sampled from density proportional to its exponential. 
   	This fact is essential for simulated annealing algorithm to return an approximate minimizer of the function. 
   	The techniques used to bound the expectation in the Euclidean case do not seem to carry over immediately to the manifold case. 
   	We use the localization technique in order to reduce the manifold case to the Euclidean case. 
   	Similarly, we use these techniques to relate the distributions for nearby temperatures in the schedule, another essential ingredient for the analysis of the simulated annealing algorithm. 
   	Note that \cite{MR3709716} works with Ricci curvature and thus applies to our setting with assumptions only on the Ricci curvature. 
   	However, to assert the convexity of the set of interior points far away from the boundary of a convex set, we need to use \cite{MR0309010} which relies on the positivity of the sectional curvature. 
    The proof of the theorem uses the Rauch comparison theorem for Jacobi fields, comparing Jacobi fields of manifolds with model manifolds with constant curvature. 
   	This is indeed the only place where we require sectional curvature lower bounds (instead of the Ricci curvature), and thus if we are working with the whole manifold Ricci curvature suffices for the analysis. 
    
	\subsection{Organization of the Paper}
    In \hyperref[sec:preliminaries]{Section \ref*{sec:preliminaries}}, we provide the preliminary definitions and theorems required for the rest of the paper. 
    In \hyperref[sec:algorithms]{Section \ref*{sec:algorithms}}, we present the algorithms for sampling and optimization. 
    In \hyperref[sec:main_results]{Section \ref*{sec:main_results}}, we state the main results regarding the guarantees of the sampling and optimization algorithms.
    In \hyperref[sec:Proof_Sketch]{Section \ref*{sec:Proof_Sketch}}, we provide a sketch of the proofs, going over the key ideas informally. 
    In \hyperref[sec:Needle]{Section \ref*{sec:Needle}}, we discuss needle decomposition on manifolds. 
    In \hyperref[sec:Iso_ineq]{Section \ref*{sec:Iso_ineq}}, we prove an isoperimetric inequality for geodesically strongly convex sets on manifolds with positive curvature and bound the one-step overlap in \hyperref[sec:One_Step]{Section \ref*{sec:One_Step}}. 
   	In \hyperref[sec:Conductance_Bound]{Section \ref*{sec:Conductance_Bound}}, we bound the conductance of the Geodesic walk and thus bound the number of steps needed to sample from a geodesic convex set. 
   	Finally, in \hyperref[sec:Optimization]{Section \ref*{sec:Optimization}}, we discuss the connection between optimization and sampling, bounding the number of iterations of the simulated annealing algorithm in optimizing a geodesically convex function. 
    
	
	\section{Preliminaries} \label{sec:preliminaries}
	   
	   In this section, we briefly note the required preliminary definitions from Riemannian geometry. 
	   For a more detailed account of Riemannian geometry see e.g. \cite{MR2229062} or \cite{nicolaescu2007lectures}. 
	   
	   \begin{definition}[Riemannian Manifold] \label{def:Riem_mani}
	   	Let $ M $ be a smooth manifold and $ TM  $ be its associated tangent space. 
	   	A Riemannian metric $ g $ is a symmetric positive definite $ (0,2) $ tensor field (a smoothly varying inner product at each point) on $ M $. 
	   	The pair $ \left(M,g\right) $ is then said to be a Riemannian manifold.
	   \end{definition}
	   
	   We will assume in the rest of the paper that the manifold is connected. The Riemannian metric naturally provides a way to measure lengths of paths on the manifold. Given the metric $ g_x : T_xM \times T_x M \to \br $, we can define a norm $ \abs{ v }_{x}  =  \sqrt{  g_x\left( v,v \right)  }  $.
	   Given a smooth path $ \gamma : [a,b] \to M $, define its length to be 
	   \begin{equation*}
	   \abs{ \gamma  } = \int_{a}^{b} \abs{ \gamma'(t)  }_{ \gamma(t) } dt . 
	   \end{equation*}
	   
	   We can then use this notion to define a metric space structure on the manifold. 
	   Given any two points $ x $ and $ y $, define 
	   \begin{equation*}
	   d(x,y) = \inf_{ \gamma(0) = x , \gamma(1) = y  } | \gamma| .
	   \end{equation*} 
	   
	   With this metric space structure, we can define the notion of metric balls around points in the manifold. 
	   Denote by $ B(x , a) $ the geodesic ball around $ x $ of radius $ a $. 
	   
	   In order to talk about variations as we move along a manifold, we need a way of talking about variations across different tangent spaces and talk about variations of vector fields relative to one another.
	   This is captured by the notion of an affine connection. 
	   Affine connections on manifolds can be quite arbitrary and need not behave well the additional Riemannian structure.  
	   Given a Riemannian structure, we can define a canonical connection which respects the Riemannian structure. 
	   This is captured by the following definition. 
	   
	   \begin{definition}[Levi-Civita Connection] \label{def:connection}
	   	Let $ \left( M,g  \right) $ be a Riemannian manifold. An affine connection is said to be the Levi-Civita connection if it is torsion-free i.e. 
	   	\begin{equation*}
	   	\nabla_{X} Y - \nabla_{Y}X = \left[X,Y\right]
	   	\end{equation*}
	   	for every pair of vector fields $ X,Y $ on $ M $ and preserves the metric i.e 
	   	\begin{equation*}
	   	\nabla g =0. 
	   	\end{equation*} 
	   	Here $ [X,Y] $ denotes the Lie bracket between two vector fields. 
	   \end{definition}
	   
	   In order to extend the notion of convexity, we would like to generalize the notion of straight lines to the Riemannian case. 
	   We can do this in two ways. 
	   One is to note that in the Euclidean case, the lines are curves that minimize the path length. 
	   The second is to note that a particle traveling along a line at constant speed does not experience any acceleration. 
	   It can be shown that these two notions are indeed equivalent and points to the fact these definitions are indeed natural extensions of lines in the Euclidean space. For a more detailed treatment of these see \cite{nicolaescu2007lectures}, \cite{MR2229062} or \cite{vishnoi2018geodesic}. 
	   
	   \begin{definition}[Geodesic]
	   	Let $ M $ be a smooth manifold with an affine connection (see \autoref{def:connection}) $ \nabla $. A smooth curve $ \gamma : I \to M $ is a said to be a geodesic if 
	   	\begin{equation*}
	   	\nabla _ { \dot { \gamma } } \dot { \gamma } = 0
	   	\end{equation*}
	   	where $ \dot{\gamma} = \frac{d\gamma \left(t\right)}{dt} $. 
	   \end{definition}

	   We next define notions of curvature on manifolds. 
	   
	   \begin{definition}[Riemannian Curvature Tensor]
	   	Let $ \left(M,g\right) $ be a Riemannian manifold with the associated Levi-Civita connection $ \nabla $. 
	   	Define the Riemannian curvature tensor as 
	   	\begin{equation*}
	   	R_M(u,v)w =\nabla_u\nabla_v w - \nabla_v \nabla_u w - \nabla_{[u,v]} w.
	   	\end{equation*}
	   	Here $ u, v $ are vector fields on the manifold. 
	   \end{definition}
	   
	   \begin{definition}[Ricci Curvature]
	   	Let $ \left(M,g\right) $ be a Riemannian manifold with the associated Riemannian curvature tensor $ R $. Then, define the Ricci curvature tensor $ Ric_M $ to be the trace of the linear map 
	   	\begin{equation*}
	   	X \to R(X,Y)Z. 
	   	\end{equation*}
	   	Here $ X , Y, Z $ are vector fields on the manifold. 
	   \end{definition}
	   
	   \begin{definition}[Sectional Curvature]
	   	Let $ \left(M,g\right) $ be a Riemannian manifold with the associated Riemannian curvature tensor $ R $. Then, define the sectional curvature $ Sec_M $ as  
	   	\begin{equation*}
	   	Sec_{M}  \left( u , v \right) = \frac { \langle R \left( u , v \right) v , u \rangle } { \langle u , u \rangle \langle v , v \rangle - \langle u , v \rangle ^ { 2 } }. 
	   	\end{equation*}
	   	where $ u , v $ are linearly independent vectors on the tangent space at a point on the manifold. 
	   	Note that the above definition depends just on the span of the vectors and not on the vectors themselves. 
	   \end{definition}
	   
	   \begin{definition}[Riemannian Volume] \label{def:vol_form}
	   	Let $ \left( M , g  \right) $ be an orientable Riemannian manifold. Then, then there exists a natural volume form on the manifold which is given in local coordinates as 
	   	\begin{equation*}
	   	d \mathrm{Vol} = \sqrt{  \det \left(g\right)  }  dx^{1} \land \dots \land dx^{n}.  
	   	\end{equation*}
	   \end{definition}
	   
	    Next, we begin by extending the notion of convexity to the setting of manifolds. 
	   Since the notion of geodesics is the natural manifold generalization of the Euclidean notion of straight lines, a natural requirement for a set on manifolds to be convex would be to require them to contain the geodesics between every two points in the set. 
	   A technicality that arises with this definition is that in general there could be multiple geodesics between two points. As an example one could consider the sphere, where there are two geodesics joining any two points. 
	   This leads to several notions of geodesic convexity depending on different requirements on the geodesics. 
	   Below we note the various definitions of geodesically convex sets. 
	   These notions vary based on the requirements on the geodesics in the set. 
	   
	   \begin{definition}[Geodesic Convex Set, see \cite{MR2229062}]
	   	Let $ \left(M,g\right)  $ be a complete Riemannian manifold. A subset $ A \subseteq M $ is said to be 
	   	\begin{itemize}
	   		\item Weakly geodesically convex if for each pair $ p, q \in A$, there exists a geodesic $ \gamma \subseteq A $ that is the unique length minimizer in $ A $ connecting $ p $ and $ q $.
	   		\item Geodesically Convex if for each pair $ p, q \in A$, there exists a geodesic $ \gamma \subseteq A$ that is the unique length minimizer in $ M $ connecting $ p $ and $ q $.
	   		\item Strongly geodesically convex if for each pair $ p, q \in A$, there exists a geodesic $ \gamma \subseteq A $ that is the unique length minimizer in $ M $ connecting $ p $ and $ q $ and $ \gamma  $ is the unique geodesic in $ A $ joining $ p $ and $ q $.
	   	\end{itemize}
	   \end{definition}
	   
	    As an example of strongly convex sets on manifolds, consider spherical caps contained in a hemisphere on the sphere in $ n $ dimensions. 
	   In general it can be shown that on manifolds, geodesic balls of small enough radius are strongly convex, as could be expected from intuition. 
	   We will work with the notion of strongly convex sets. 
	   We do not consider the notion of totally convex sets (see \cite{vishnoi2018geodesic}) here, since in the setting of positively curved manifolds, this class is not sufficiently rich. 
	   For example, spherical caps are not totally convex.
	   It can be seen that the notion of strongly convex sets on the Euclidean space corresponds to the usual notion of convex sets, since straight line segments are the unique minimizing geodesics in the Euclidean space.
	   This should not be confused with other notions of strong convexity of sets and functions in the Euclidean space. 
	   
	   	\begin{definition}[Geodesically Convex Function]
	   	Let $ \left(M,g\right) $ be a Riemannian manifold and let $ K \subseteq M $ be a strongly convex subset. 
	   	Then, a function $ f : K \to \br $ is said to be geodesically convex if for every pair $ p, q \in K$, and for any geodesic $ \gamma : \left[0,1\right] \to K $ joining $ p $ and $ q $, we have that $ f \circ \gamma $ is convex, that is, for all $ t \in \left[0,1\right] $,
	   	\begin{equation*}
	   	f\left( \gamma \left(t\right)  \right) \leq \left(1-t\right) f\left( p \right) + t f\left(q\right). 
	   	\end{equation*} 
	   \end{definition}
	   
	   As in the case of the Euclidean space, we can characterize geodesically convex functions in terms of first and second order conditions. 
	   In fact, the characterizations below follow from the Euclidean versions by restricting to the geodesic of interest. 
	   
	   \begin{theorem} [see \cite{MR1326607}]
	   	Let $ \left( M , g  \right) $ be a Riemannian manifold and let $ K \subset M  $ be a strongly convex set. Let $ f : K \to \br $ be a function. 
	   	\begin{itemize}
	   		\item If $ f $ is differentiable, then for the geodesic $ \gamma $ joining points $ p $ and $ q $, we have 
	   		\begin{equation*}
	   		f \left( p \right) + \dot { \gamma } \left ( f \right ) \left ( p \right ) \leq f \left ( q \right )
	   		\end{equation*}
	   		where $ \dot { \gamma } \left ( f \right ) $ denotes the derivative of $ f $ along $ \gamma $. 
	   		\item If $ f $ is twice differentiable, then for the Hessian defined as 
	   		\begin{equation*}
	   		Hess_f \left(X,Y\right) = \ip{ \nabla _ { X } \nabla f } {Y},
	   		\end{equation*}
	   		for two vector fields $ X $ and $ Y $, we have $ Hess_{f} $ is positive definite. 
	   	\end{itemize}
	   \end{theorem}

	   In order to sample from a distribution, we use the general technique of Markov Chain Monte Carlo on continuous state space. 
	   A Markov chain is a stochastic process where the distribution of the next step depends only on the current state and not the history of the process. 
	   Formally, 
	   \begin{definition}[Markov Chain]
	   	Let $ \left( \Omega , \mathcal{F}  \right) $ be a measurable space. 
	   	Let $ Q_0$ be a measure on $ \left( \Omega , \mathcal{F}  \right) $.
	   	A time homogeneous Markov chain is specified by a probability measure $ P_w $ for each $ w \in \Omega $. 
	   	Define the chain by the sequence of measures $ Q_i $, given by 
	   	\begin{align*}
	   	Q_{i+1}\left( A \right) = \int P_{u} \left(A\right) dQ_{i} \left(u\right)
	   	\end{align*}
	   \end{definition}
	   
	   The probability measure $ P_x  $ is said to be the one step of the above chain from the point $ x $. 
	   By slight abuse of notation, we use $ P_x \left( y \right) $ for $ P_x  \left( \{y \}  \right) $, whenever $ y  $ is a point on the manifold. 
	   
	   An important notion in the study of Markov chains is that of a stationary distribution. 
	   A stationary distribution of a Markov chain is a distribution on the state space such that the distribution after one step, starting with this distribution, remains the same. 
	   Formally, 
	   \begin{definition}[Stationary Distribution]
	   	Let $ P_x $ be the transition operator at point $ x $ of a Markov chain on the state space $ \Omega $. 
	   	A measure $ \pi $ on $ \Omega $ is said to be stationary if for every measurable subset $ A $, we have 
	   	\begin{equation*}
	   	\int_{\Omega} P_x \left(  A \right) d\pi\left(x\right) = \pi\left(A\right).
	   	\end{equation*}
	   \end{definition}
	   
	   It can be shown under reasonable assumptions that the distribution of a Markov chain converges to its stationary distribution. 
	   This fact forms the basis for the use of Markov chains in sampling algorithms. 
	   The idea is to set up a chain such that the desired distribution is stationary and run the chain until the distribution is close to the stationary distribution. 
	   Below we note a property that implies stationarity and will also be helpful property when transforming a Markov chain with a certain stationary distribution to a Markov chain with a different stationary distribution. 
	   
	   \begin{definition}
	   	A Markov chain is said to be time-reversible with respect to a measure $ \pi $ if for any two measurable sets $ A, B $ 
	   	\begin{equation*}
	   	\int_B P_u \left(A\right) d\pi\left(u\right) = \int_A P_u \left(B\right) d \pi(u ) . 
	   	\end{equation*}
	   \end{definition}
	   From the above definition, it is clear that if a Markov chain is time-reversible with respect a certain measure then that measure is a stationary distribution for the chain. 
	   
	   One of the primary methods used to bound the mixing time is to bound the conductance of the walk. Intuitively, an impediment to fast mixing is the existence of large sets from which walk has a low chance of leaving, forming a bottleneck.
	   The notion of conductance makes this notion rigorous.
	   It can be shown that the conductance is tightly linked to the spectral gap of a Markov chain, which also controls the rate of mixing. 
	   Cheeger's inequality relates these notions, and is widely used in spectral graph theory. 
	   Cheeger's inequality was in fact originally proven in the setting of differential geometry, relating the Cheeger isoperimetric constant of a manifold to the eigenvalues of the Laplacian on the manifold.  
	   
	   \begin{definition}[Conductance of a Random Walk] \label{def:conductance}
	   	Let $ P_x $ be the transition operator at point $ x  $ of a Markov chain on the state space $ \Omega $. 
	   	Define the ergodic flow with respect to a distribution $ \pi  $ of a set $ A $ to be 
	   	\begin{equation*}
	   	\Phi\left(A\right) = \int_{x \in A} P_{x}\left( \Omega \backslash A   \right) d \pi\left(x\right).
	   	\end{equation*}
	   	Then, the conductance of the walk is defined to be 
	   	\begin{equation*}
	   	\phi = \min_{\pi\left(A\right) \leq 1/2}  \frac{\Phi\left(A\right)}{\pi\left(  A \right)  }.
	   	\end{equation*}
	   	Similarly, for $ 0 \leq s \leq 0.5  $, we can define the $ s $-conductance to be 
	   	\begin{equation*}
	   	\phi_s= \min_{s < \pi\left(A\right) \leq 1/2}  \frac{\Phi\left(A\right)}{\pi\left(  A \right) - s }.
	   	\end{equation*}
	   	For each point $ x \in \Omega $, define the local conductance $ \ell\left(x\right) $ to be the probability of moving away from the point $ x $ i.e. 
	   	\begin{equation*}
	   	\ell\left(x\right) = 1- P_{x} \left(x \right). 
	   	\end{equation*}
	   \end{definition}
	   In order to measure the notion of a good start for the random walk, we use the following standard notions. 
	   %
	   \begin{definition}[Warm start and $L_2$ distance, see \cite{vempala2005geometric}] \label{def:Warm_Start}
	   	Let $ \mu_1 $ and $ \mu_2 $ be two measures with common support. 
	   	\begin{itemize}
	   		\item We say that $ \mu_0 $ is $ H $-warm with respect to $ \mu_1 $ if 
	   		\begin{equation*}
	   		\sup_{A} \frac{\mu_0 \left(A\right)}{\mu_1\left( A\right)} \leq H. 
	   		\end{equation*}
	   		\item We define the $ L_2 $ distance of $ \mu_0 $ with respect to $ \mu_1 $ as 
	   		\begin{equation*}
	   		\norm{ \mu_0 / \mu_1   } = \int \left(  \frac{d \mu_0 }{d \mu_1}  \right)^2 d \mu_1  = \int  \left(\frac{d \mu_0 }{d \mu_1} \right) d \mu_0 . 
	   		\end{equation*}
	   	\end{itemize}
	   	
	   \end{definition}

		\todo[disable]{Add some comments on Geodesically convex sets, examples and closure conditions. 
		Totally convex sets}
		

		\subsection{Oracles} 
		For the problem to be well-posed and useful, we need to specify how the algorithm accesses the manifold. 
		In our setting, we also assume that our algorithm has access to an oracle that takes as input a point on the manifold and real vector $ v $ and outputs $ \exp_x \left( v \right) $. 
		Here, we are associating the tangent space $T_x$ at the point $ x $ with $ \br^n $. 
		In our algorithms we will need access only to uniformly sampled vectors, which are invariant under orthogonal transformations, in the tangent space, hence the choice of the basis of the tangent space does not affect the computation.
		The oracle for the exponential map is common practice in the literature on optimization on manifolds, especially ones using first order methods of optimization, for example see \cite{zhang2018towards}. 
		\cite{mangoubi2016rapid}, \cite{lee2017geodesic} and \cite{lee2018convergence} also present their algorithm  in the presence similar oracles. For the particular manifolds of interest they give methods to construct the required oracle.
		The computation of the exponential map amounts to solving systems of ordinary differential equations, given by the characterizing equations of geodesics. 
		See \cite{lee2017geodesic} for a discussion on the method of construction of the oracle. 
		For natural matrix manifolds the exponential map is well-understood and computationally easier to handle; for example, see \cite{MR2364186}.


	 \section{Stationary Distribution of the Geodesic Walk}
	In this section, we show that the stationary distribution of the geodesic random walks on a convex set on a manifold is the uniform distribution on the convex set. 
	This fact follows from using the well known Liouville Theorem for geodesic dynamics (see \cite{MR2229062}), and was observed in \cite{mangoubi2016rapid} for the case of the unconstrained walk on a manifold of positive curvature.
	
	The geodesic walk on the manifold is also a reversible Markov chain.
	This claim was proved in the more general case of Hamiltonian Monte Carlo in \cite{lee2018convergence}, but we present a proof for completeness. 
	The proof as in the case of the stationarity uses the Liouville theorem. 
	
	\todo[disable]{We have not defined reversibility. In a way it's fine, but given that we are even defining Markov chains we ought to define it and say something about its significance.}
	
	\begin{theorem}
		The geodesic walk is a reversible Markov chain. 
	\end{theorem}
	\begin{proof}
		We show this in the case of unrestricted (i.e., without a metropolis filter to restrict the walk to a convex subset) geodesic walk in compact manifolds; the reversibility of the restricted version then follows from the analysis of the Metropolis filter, for example see Lemma 1.9 in \cite{lovasz1993random}.  \todo[disable]{doesn't that last claim need some more justification?}
		In order to show the reversibility, it suffices to show that the transition operator of the walk is self-adjoint. 
		Again, we restrict to the case of fixed step length. 
		The general case where the step lengths are picked independently in each step is similar. \todo[disable]{In the previous sentence it's not clear what the general case means. In the following we should say what f and g are.}
		Thus consider 
		\begin{equation*}
		\int_{x \in M } \int_{ u \in S_x }  f\left( \exp_x \left( \delta u \right)    \right)  g(x)  \; d \mu_x(u) d \mathrm{Vol} \left(x\right)  
		\end{equation*}
		where $ \mu_x $ is the uniform distribution on the unit sphere $S_x$ in the tangent space $T_x$.   
		Let us consider the above integral as an integral over the Liouville  measure $ \mathcal{L} $. 
		\begin{equation*}
		\int f\left( \exp_x \left( \delta u \right)    \right) g(x) \; d\mathcal{L} \left(x,u\right) . 
		\end{equation*}
		Now note that the joint distribution of $ \left(   \left(x,u \right) ,   \left( \exp_x\left( \delta u  \right) , \bar{u} \right)   \right) $ is the same as $  \left(   \left(\exp_y \left( - v   \right) , \bar{v} \right),  \left( y , - v \right)   \right)  $, where $ y $ is drawn from the uniform measure on the manifold and $ v $ is drawn from the unit sphere on the tangent space $ T_{y} $. \todo[disable]{$T_y$?} 
		Here $ \bar{u} $ denotes the derivative of the geodesic at the point $  \exp_x\left( \delta u  \right) $. 
		Making the above change of variables, we get the required theorem.
		To see this note that $  \left(\exp_y \left( - v   \right) , -\bar{v} \right)$ is obtained by applying the geodesic flow to $ \left( y , -v  \right) $. 
		From Liouville theorem and the symmetry of the uniform distribution on the unit sphere, we get that this has the same distribution as $ \left( x , u\right) $. 
		The required claim now follows by change of variables. 
	\end{proof}
	
	With the reversibility given above, we see that the stationary distribution of the geodesic walk is the uniform measure as required. 
	We note this in the following theorem. \todo[disable]{Why is the following theorem being credited to [MS]? They don't talk about restrictions to g-convex sets.}
	
	\begin{theorem}
		The stationary distribution of the geodesic walk on a geodesically convex set $ K $ is the uniform distribution on $ K $. 
	\end{theorem}

    \section{Algorithms} \label{sec:algorithms}
    
    In this section, we first describe an algorithm to sample from a convex set on a manifold, given access to the membership oracle for the convex set $ K $. 
    We consider a Markov chain whose stationary distribution is the uniform distribution on the convex set. 
    The algorithm in each step samples from the ball uniformly on the tangent space at the current point and moves along the induced geodesic on the manifold for a length specified by the step length. 
    If the geodesic takes us out of the convex set, then we reject the step and resample from the tangent space.
    
    \begin{algorithm}[H]
    	\SetAlgoLined
        \caption{Geodesic Walk on a Manifold $ M $}
        \label{alg:GeoWalk}
        \DontPrintSemicolon
        \KwIn{Dimension $ n $, Convex set $ K $, Starting Point $ X_0 \in K $, Step Size $ \delta $, Number of Steps $ N $.}
        
        \For{$ \tau \leq N $}{
             Pick $ u_{\tau+1} \leftarrow N\left( 0 ,I   \right)  $ in $ T_{X_{\tau}}  M $. \;
             \eIf{$  \exp_{X_{\tau}}  \left( \delta u_{\tau+1}  \right) \in K $}{
                 Set $ X_{\tau+1} \leftarrow \exp_{X_{\tau}}  \left( \delta u_{\tau+1}  \right) $. \;
             }{
              Set $ X_{\tau+1} \leftarrow X_{\tau} $. \;
             }
        }
        \KwOut{Point $ X_N$ approximately uniformly sampled from $K$.} 
    \end{algorithm}

    We work with Riemannian manifolds that are complete as metric spaces. The Hopf--Rinow theorem asserts that for such manifolds the
    exponential map is defined on the whole tangent space. Thus all the oracle calls in \hyperref[alg:GeoWalk]{Algorithm \ref*{alg:GeoWalk}} are well-defined. 
    
    
    Variations of the walk above have been considered in \cite{lee2017geodesic}, \cite{lee2018convergence} and \cite{mangoubi2016rapid}. 
    Hamiltonian walk with the Hamiltonian given by $ H( x, v  ) = \frac{1}{2} \ip{v}{v}_{g^{-1}\left(x\right)}$ leads to a variation of the geodesic walk because of the well known fact that the geodesic flow is the Hamiltonian flow for the above Hamiltonian. 
   Here $ g $ refers to the Riemannian metric as defined in \hyperref[def:Riem_mani]{Definition \ref*{def:Riem_mani}}. 
   Note that executing the geodesic walk needs oracle access to the exponential map and a membership oracle to the convex set. 
   
   Towards the goal of optimizing convex functions, we adapt the simulated annealing algorithm from the Euclidean setting to the Riemannian setting. 
   Given a function $ f $ and a ``temperature" $ T $, define the probability density $ \pi_{f,T} \sim  e^{ -\frac{f}{T}  }$.
   Intuitively, the function puts more weight on points where the function attains small values and sampling from the distribution is likely to output points near the minima of the function $ f $ for low enough temperature $ T $. 
   The issue is that sampling from the distribution for a low enough temperature is a priori as hard as solving the initial optimization problem. 
   The way to get around this issue is to set a temperature schedule in which one progressively ``cools" the distribution such that the sample from the previous temperature acts to make it easier to sample from the next temperature. 
   Once we reach a low enough temperature, the sample we attain will be close to the optimum point with high probability. 
   
   Since we need to sample from a distribution proportional to the $ e^{-\frac{f}{T}}  $, the natural idea would be to use the Metropolis filter with respect to the original uniform sampling algorithm. 
   This leads to the following algorithm for sampling from the required distribution.

   
    \begin{algorithm}[H]
       \SetAlgoLined
       \caption{Adapted Geodesic Walk on a Manifold $ M $ and function $ f $.}
       \label{alg:Geowalk_LC1}
       \DontPrintSemicolon
       \KwIn{Dimension $ n $, Convex set $ K $, Convex function $ f : K \to \br$, Starting Point $ X_0 \in K $, Step Size $ \eta $, Number of Steps $ N $.}
       \For{$ \tau \leq N $}{
           Pick $ u_{\tau+1} \leftarrow N\left( 0 ,I   \right)  $ in $ T_{X_{\tau}}  M $. \;
           \eIf{$ y = \exp_{X_{\tau}}  \left( \eta \, u_{\tau+1}  \right) \in K $}{
               With probability $ \min \left( 1 , e^{-f( y  ) + f(X_{\tau})  } \right) $, set $ X_{\tau+1} \leftarrow y  $. \;
               With the remaining probability, set $  X_{\tau+1} \leftarrow X_{\tau} $. \;
           }{
               Set $ X_{\tau+1} \leftarrow X_{\tau} $. \;
           }
       }
       \KwOut{Point $ X_N \in K $ sampled approximately proportional to $ e^{-f} $.}
   \end{algorithm}
   
    Given the algorithm for sampling at fixed temperature, we adapt the simulated annealing algorithm for the case of positively curved manifolds. 
   The sequence of temperatures that the optimization takes is called the temperature schedule or the cooling schedule. 
   
   \begin{algorithm}[H]
       \SetAlgoLined
       \caption{Simulated Annealing on Manifold $ M $.}
       \label{alg:Optimization1}
       \DontPrintSemicolon
       \KwIn{Dimension $ n $, Convex set $ K $, Convex function $ f : K \to \br $, Starting Point $ X_0 \in K $, Number of Iterations $ N $, Temperature Schedule $ T_i $.} 
       \For{$ \tau \leq N $}{
           Sample $ X_{\tau }  $ according to distribution $ \pi_{f, T_{\tau}} $  using \hyperref[alg:Geowalk_LC1]{Algorithm \ref*{alg:Geowalk_LC1}} for  $ \pi_{f, T_{\tau}} $ with starting point $ X_{\tau -1} $. \;
       }
       \KwOut{Point $ X_N \in K$ that approximately minimizes $ f $.}
   \end{algorithm}


   The main thing to specify in the design of the algorithm above is the temperature schedule, that is, the sequence of temperatures $ T_i $ from which we sample. 
   The aim is to set the schedule such that each temperature is close enough to such that the distributions are similar, while still maintaining a small number of temperature updates. 
   Following \cite{kalai2006simulated}, we show that 
   \begin{equation*}
   T_{i+1} = \left(1 - \frac{1}{\sqrt{n}}\right) T_i
   \end{equation*}
   provides us with the required guarantees. 
   We set the initial temperature $ T_0 $ such that the uniform distribution on convex set is close in the $ L_2 $ norm to the initial distribution.
   Note that setting  $ T_0 = \max_{x} f(x)$ satisfies this requirement. 
   Since the function is $ L $-Lipschitz, we can bound this by $ T_0 \leq \min_{x} f(x) + LD $ where $ D $ is the diameter of the convex set. 
   
   \section{Main Results} \label{sec:main_results}
   In this section, we discuss the main results of the paper. 
   The first result is regarding sampling from a convex set in manifolds with positive curvature. 
	\todo[disable]{Make notation about Riemann tensor clear}
   \begin{restatable}{theorem}{sampling}[Mixing Time Bound] \label{thm:sampling}
		Let $ \left( M , g  \right) $ be an $ n $-dimensional Riemannian manifold and $ K \subset M $ be a strongly convex subset satisfying the following conditions:
		\begin{itemize}
			\item $ M $ has non-negative sectional curvature i.e. $ Sec_{M} \geq 0 $.
			\item The Riemannian curvature tensor is upper bounded i.e. $ \max \norm{R_M}_{F}  \leq R $. 
			\item $ K $ contains a ball of radius $ r $. 
			\item $ K $ has diameter $ D $. 
		\end{itemize}
		Let the starting distribution be an $ H $-warm start (see \hyperref[def:Warm_Start]{Definition~\ref*{def:Warm_Start}}). Then, in  
		\begin{equation*}
		t = O \left( \frac{H^2 D^2 n^3 \left( R+1\right)   }{r^2 \epsilon^2} \log \left(  \frac{H}{\epsilon}  \right) \right)
		\end{equation*}
		steps, the distribution of $ X_t $, the output of \hyperref[alg:GeoWalk]{Algorithm~\ref*{alg:GeoWalk}} is $ \epsilon $-close to the uniform distribution on convex set in total variation. 
   \end{restatable}
   
   From the definition, we get that $ \abs{Sec_{M} } \leq Rn^{-1} $ implies $ \norm{R_M}_F  \leq R  $. So, the above theorem can be stated in terms of sectional curvature as well. 
    See \hyperref[sec:Conductance_Bound]{Section \ref*{sec:Conductance_Bound}} for the proof of the above theorem.
   We next state the theorem regarding the optimization of convex functions on convex sets. 
   

   \todo[disable]{Restate the theorem with correct references.}
   
   \begin{restatable}{theorem}{simulatedannealing}[Simulated Annealing] \label{thm:simulated_annealing}
   		Let $ \left( M,g\right) $ be a manifold with non-negative Ricci curvature. Let $ K \subseteq M $ be a strongly convex set satisfying the requirements in \hyperref[thm:sampling]{Theorem~\ref*{thm:sampling}}. Let $ f : K \to \br $ be a geodesically convex function with Lipschitz constant $ L $. Then, starting from a uniform sample from $ K $, \hyperref[alg:Optimization1]{Algorithm~\ref*{alg:Optimization1}} runs for 
   		\begin{equation*}
   		O \left( \frac{D^2 n^{7.5} \left(R+1\right)  L^2}{r^2 \epsilon^2 \delta^6} \log \left( \frac{n}{\delta} \log\left( \frac{T_0 \left(n+1\right) }{\epsilon \delta}   \right)        \right)    \log^5\left( \frac{T_0 n }{\epsilon \delta}   \right)    \right) 
   		\end{equation*}
   		steps and with probability $ 1 - \delta $ outputs a point $ x^{*} $ such that 
   		\begin{equation*}
   		f\left( x^{*}  \right) - \min_{x} f(x) \leq \epsilon.  
   		\end{equation*}
   \end{restatable}
The requirement to start with the uniform distribution is made for simplicity and can work with a warm start.
    \hyperref[sec:Optimization]{Section \ref*{sec:Optimization}} contains a proof of the above theorem.
   Note that the dependence on the probability of error can be reduced from polynomial to logarithmic using standard error reduction 
   by running independent trials of the algorithm and outputting the minimum value across the trials.  
   
    \section{Proof Sketch} \label{sec:Proof_Sketch}
    The dominant strategy for showing that geometric Markov chains mix fast in convex bodies in the Euclidean space goes via the notion of conductance of Markov chain (\hyperref[def:conductance]{Definition \ref*{def:conductance}}); see e.g. \cite{vempala2005geometric}. Intuitively, the conductance of a Markov chain is a measure of how well-connected the Markov chain is: small conductance implies that there are sets where the walk can get stuck and takes a long time to get out, somewhat like the barbell; large conductance implies that such bottlenecks do not exist.
    We will adapt this strategy to the manifold setting and show that the conductance of our walk is large, implying that the walk mixes fast. 
    Lower bounding the conductance, as in the case of the Euclidean space, reduces to showing isoperimetric inequalities and bounding the overlap of one-step distributions of the chain. 
    The key tool here is the localization lemma \cite{lovasz1990mixing, kannan1995isoperimetric}.
    It is useful for proving many geometric inequalities and in particular isoperimetric properties of convex sets in the Euclidean space.
Recently, the classical localization lemma which applies to sets in $\brn$ was extended to Riemannian manifolds~\cite{MR3709716} to prove isoperimetry theorems on manifolds.  
    
    \begin{theorem}[Informal statement of \autoref{thm:isoperimetry}] \label{thm:isoperimetry_informal}
        Given a partition of a strongly geodesically convex set in a Riemannian manifold (with non-negative Ricci curvature) into three subsets $K_1$, $K_2$ and $ K_3 $ such that $ K_1 $ and $ K_3 $ are well separated, then the volume of $ K_2 $ is large compared to the volume of the minimum of the two separated sets.  
    \end{theorem}
    
    The geodesic walk on the convex set is geometric in nature and thus its conductance is tightly linked with the isoperimetry of the base space. To prove the rapid mixing of the geodesic random walk, we first show:
    
    \begin{theorem}[Informal statement of \autoref{thm:one_step}] \label{thm:one_step_informal}
        For two nearby points on the convex set, the one-step distributions of the geodesic walk from these points are close in total variation distance. 
    \end{theorem}
    
    Next, for any subset $S_1$ of the strongly geodesically convex set $K$ (and its complement in $K$, namely $S_2$), we consider a subset $S'_1$ from which the geodesic walk is unlikely to escape. 
    Using \autoref{thm:one_step_informal}, the aim is to show that $S'_1$ and the corresponding set $S'_2$ in the complement are well separated, thus reducing the conductance to the isoperimetry claim of \autoref{thm:isoperimetry_informal}. One issue with this approach is that there could be a point in the set with very low probability of leaving that point. 
    To deal with this, we restrict to the subset of the convex set which only consists of points with high probability of taking a step away from the current point. 
    Using a theorem of Cheeger and Gromoll it can be shown that this subset is convex. Moreover, we show with another use of localization that it has volume comparable to that of the original set.
    
    \begin{theorem}[Informal statement of \autoref{thm:rev_iso}]
        The set of points with high local conductance is a convex subset of the original convex set and large fraction of the size of the original set. 
    \end{theorem}
    
    Using the above, we can show that the conductance of the walk is high, thus showing that the above Markov chain mixes rapidly. 
    Running the Markov chain for sufficiently many steps, gives us a distribution that is close in total variation distance to the stationary distribution, which in this case is the uniform distribution as required. 
    
    \begin{theorem}[Informal statement of \autoref{thm:sampling}]
        Given a warm start, \hyperref[alg:GeoWalk]{Algorithm~\ref*{alg:GeoWalk}} outputs a point that is approximately sampled from the uniform distribution on a geodesically convex set in time polynomial in the dimension. 
    \end{theorem}
    
    Towards showing the reduction from sampling to optimization on manifolds, we adapt the algorithm from \cite{kalai2006simulated}
    for the Euclidean setting.
    We first show that we can sample from the required distributions for each temperature by adapting the algorithm to the case of sampling from the required log-concave distributions. 
    
    \begin{theorem}[Informal statement of \autoref{thm:sampling_logconcave}]
        Given a starting distribution that is close in the $ L_2 $ distance, \hyperref[alg:Geowalk_LC1]{Algorithm~\ref*{alg:Geowalk_LC1}} outputs a point that is approximately sampled from a log-concave distribution in polynomial time. 
    \end{theorem}
    
    One then needs to pick a temperature schedule satisfying two seemingly opposite requirements. Firstly, we require that the schedule makes large enough updates so as to not require too many samples in order to reach the required low temperature. 
    Secondly, we require that the temperatures are close enough that the two distributions we get for these temperatures are close by in $ L_2  $ distance. 
    We pick a temperature schedule that multiplicatively updates the temperature in each iteration i.e. we pick $ T_{i+1} =  \left( 1 - n^{-0.5}  \right)   T_i    $.
    This ensures that the temperatures are rapidly decreasing. 
    But, it remains to be shown that the temperatures are indeed close by in $ L_{2} $ distance. 
    The Euclidean version of this theorem is usually shown by constructing an auxiliary function whose log-concavity given the bound. 
    The log-concavity is established using the fact that the marginals of log-concave measure is log-concave. 
    The above strategy does not generalize to the manifold setting. 
    We use localization directly to show the required inequality by reducing to the Euclidean case. 
    
    \begin{theorem}[Informal statement of \autoref{thm:adjacent_dist}]
        For two adjacent temperatures in the schedule, the two distributions are close in $L_2$ distance. 
    \end{theorem}
    
    Once we can sample from the distribution for each temperature, we need to then show that for low temperatures we get points near the optimal with good probability. 
    The equivalent result in the Euclidean case is shown by reducing general convex functions to linear functions and for linear functions reducing to the case where the convex set is a cone. 
    Again, this approach doesn't seem to generalize to the setting of manifolds. So, we use the localization technique using a more general version stated with respect to a guiding function. 
    The distance from the minimum is used as the guiding function, leading to needles that pass through the minimum of the function, and then we can reduce to the inequality shown in \cite{kalai2006simulated}. 
    
    \begin{theorem}[Informal statement of \autoref{thm:low_temperature_expectation}]
        For small enough temperature, the expected value of the function sampled from the distribution is close to the minimum value of the function. 
    \end{theorem}
    
    Given the above theorems, we can show the guarantees of the algorithm for optimizing convex functions on strongly convex sets on manifolds. 
    
    \begin{theorem}[Informal statement of \autoref{thm:simulated_annealing}]
        Given a geodesically convex function $ f $ on a geodesically convex set, \hyperref[alg:Optimization1]{Algorithm \ref*{alg:Optimization1}} finds a point that is $ \epsilon $-close to the optimal point running in time polynomial in the dimension and $ \epsilon^{-1} $. 
    \end{theorem}
    
       \section{Needle Decomposition on Riemannian Manifolds} \label{sec:Needle}
    We first recall the localization lemma for the Euclidean space mentioned earlier.
    
    \begin{theorem}[Localization Lemma, \cite{lovasz1990mixing, kannan1995isoperimetric}]
    	Let $ g $ and $ h $ be real-valued lower semi-continuous measurable function on $ \br^n $ satisfying 
    	\begin{align*}
    	\int_{\brn } g(x) dx > 0
    	\intertext{and} 
    	\int_{\brn } h(x) dx > 0.
    	\end{align*}
    	Then, there exists a linear function $ \omega : \left[ 0,1 \right] \to \br_+  $ and two points $ a , b  \in \brn$ such that 
    	\begin{align*}
    	\int_{0}^{1} \omega^{n-1} (t)   g\left( a t + (1 - t) b  \right) dt > 0 
    	\intertext{and}
    	\int_{0}^{1} \omega^{n-1} (t)   h\left( a t + (1 - t) b  \right) dt > 0 . 
    	\end{align*}
    \end{theorem}
    
    The proof of the above theorem follows the strategy of repeated bisection.
    We find a halfspace for which both the given inequalities are satisfied. 
    \todo[disable]{We find a hyperplane or halfspace? Both sides or one side?} 
    Taking this to the limit, we are left with one dimensional interval such that proving the inequality on the original convex set  reduces to proving the inequality on the interval with respect to the measure we get from the limiting process. 
    See \cite{MR3709716} for an overview of the development of the localization technique. 
    
    In the Riemannian setting, it is not immediately clear what the right analogue of this bisection argument is.
    In a recent work \cite{MR3709716}, a Riemannian analogue of the localization lemma was proven by invoking connections to optimal transport. This localization lemma produces a disintegration of measure for the volume form on the manifold into measures $ \mu_I $ supported on geodesic sections such that the measures $ \mu_I $ are pushforwards of measures on the real line satisfying the curvature dimension condition on the interval of support. 
    \todo[disable]{We need to comment on if these are different from what Klartag proves?} 
    
    We next define the notion of curvature-dimension conditions on a weighted manifold, which shall be central in the development of the localization technique over Riemannian manifolds; see \cite{MR3709716} for more details. \todo[disable]{What is $\rho$ in the following? We should say a bit more about it and any conditions that it might need to satisfy. I think we are missing quantifications on $N, \kappa$ etc.}
    
    \begin{definition}[Curvature-Dimension Condition]
    	Let $ \left( M , d , \mu  \right) $ be a $ n $-dimensional weighted Riemannian manifold and let $ \rho: M \to \br $ be a smooth function such that the weight measure has density $ e^{-\rho} $ with respect to the Riemannian volume measure. Let the generalized Ricci curvature with parameter $ N \in \left( - \infty ,1  \right) \cup \left[ n , + \infty   \right] $ be defined through the equation 
    	\begin{equation*}
    	Ric_{\mu, N} \left( v ,v   \right) = \begin{cases}
    	Ric_{M} \left(v,v\right) + Hess_{\rho} \left(v,v\right) - \frac{\left( \partial_v \rho  \right)^2}{N-n} & N \neq \infty, n \\
    	
    	Ric_{M} \left(v,v\right) + Hess_{\rho} \left(v,v\right) & N = \infty \\
    	Ric_{M} \left(v,v\right)  & N = n
    	\end{cases}
    	\end{equation*} 
    	for every $ x \in M $ and tangent vector $ v \in T_x M $. 
    	For $ \kappa \in \br $, we say that the weighted manifold satisfies the curvature-dimension condition $ CD\left( \kappa , N \right) $ if for  every $ x \in M $ and tangent vector $ v \in T_x M $, we have 
    	\begin{equation*}
    	Ric_{\mu , N }  \left(  v , v  \right)  \geq \kappa \cdot g(v,v) . 
    	\end{equation*}
    \end{definition}
    
    The condition defined above comes up naturally in the study of diffusion operators. The notion is used to prove hypercontractivity inequalities and logarithmic Sobolev inequalities. For more information about the connections to diffusion operators, see \cite{MR3155209}.
    This condition has also been considered in optimal transport in terms of displacement interpolation in the space of probability measures. 
    For discussion on this condition and the relationship to the above curvature dimension condition, see \cite{MR3298475}.

    When $n$ is positive it can be interpreted as a generalized upper bound on the dimension for weighted manifold and $\kappa$ can be interpreted as a generalized lower bound on the curvature of the weighted manifold.
    Manifolds in this paper always have non-negative sectional curvature and satisfy $\kappa \geq 0$. We make extensive use of 
    localization for manifolds satisfying $CD(0,N)$ condition with $N=n$ or $\infty$; for example in Theorems
    \ref{thm:isoperimetry}, \ref{thm:rev_iso}, \ref{thm:isoperimetry_lc}.
    
    \todo[disable]{In the following definition, one has to say what kind of object $\eta$ is when first mentioning it. Klartag does it in his definition of needle. We are also missing quantifiers on $\kappa, N$.}
    
    \begin{definition}[Needle, \cite{MR3709716}]
    	Let $\left( M , d , \mu  \right)  $ be a weighted Riemannian manifold satisfying the $ CD\left( \kappa ,  N  \right) $ condition for some $ \kappa \in \br $ and $ N \in \left( - \infty ,1  \right) \cup \left[ n , + \infty   \right] $.
    	Let $ \eta $ be a measure on $ M $. 
    	We say that $ \eta  $ is a $ CD(\kappa, N) $ needle if there exists an open interval $ A  \subset \br$, a smooth function $ \Psi : A \to \br  $ and a minimizing geodesic $ \gamma : A \to M $ such that 
    	\begin{enumerate}
    		\item $ \eta $ is the pushforward of the density on $ A $ that is proportional to $ e^{- \Psi } $ under $ \gamma $ i.e. for any measurable subset $ B $, $ \eta\left(B\right) = \theta\left( \gamma^{-1} B  \right) $, where $ \theta \propto e^{-\Psi} $. 
    		\item For every $ x \in A $, we have 
    		\begin{equation*}
    		\Psi'' \left(x \right) \geq  \kappa + \frac{  \left( \Psi' \left(x\right)  \right)^ 2}{N-1} . 
    		\end{equation*}
    		For the case of $ N = \infty $, we intepret the second term on the right hand side to be zero.
    	\end{enumerate}
    	If the above inequality is an equality, then the measure is said to be a $ CD\left( \kappa , N \right) $ affine needle. 
    \end{definition}
    Note that the needle condition is equivalent to the $ CD(\kappa,N) $ condition on the preimage interval. 
    As we shall see later, the geodesics play the role of the lines that appear in the Euclidean case. \todo[disable]{Note to self: check if we do this.} 
    \todo[disable]{We have nowhere said that our manifolds are allowed to have boundary and g-convex sets are submanifolds with boundary, in particular in the following theorem. This is important as I think normally one thinks of manifolds without boundary and then Yau says that we cannot have convex manifolds without boundary. This point can be made in the intro around where Yau's theorem is mentioned.}	
    
    
    \begin{theorem}[Needle Decomposition, \cite{MR3709716}] \label{thm:Klartag}
    	Let $ n \geq 2  $, $ \kappa \in \br $ and $ N \in \left( - \infty, 1 \right) \cup \left[ n , \infty   \right] $. 
    	Let $ \left( M , d , \mu  \right) $ be a weighted complete Riemannian manifold satisfying the $ CD\left( \kappa, N  \right) $ condition. 
    	Consider an integrable function $ f :  M \to \br  $ satisfying $ \int_{M} f d \mu = 0 $ and $ \int_M \abs{f(x)} d( x , x_0  )  d\mu   < \infty  $ for some $ x_0 $. 
    	Then, there exists a partition $ \Omega  $ of $ M $ and a measure $ \nu  $ on $ \Omega  $ and a family of measure $ \left\{ \mu_i  \right\}_{i \in \Omega} $ such that 
    	\begin{enumerate}
    		\item For any measurable set $ A $, 
    		\begin{equation*}
    		\mu\left( A\right) = \int_{\Omega}  \mu_i \left( A\right) d \nu(i).
    		\end{equation*}
    		This can be seen as a disintegration of measure.
    		\item For almost every $ i \in \Omega $, $ i $ is an image of a minimizing geodesic and $ \mu_i $ is supported on $ i $. Furthermore, for almost every $ i$, either $ i $ is a singleton or $ \mu_i $ is a $ CD( \kappa, N) $ needle. 
    		\item For almost every $ i $, 
    		\begin{equation*}
    		\int_i f d \mu_i = 0.
    		\end{equation*}
    	\end{enumerate}
    \end{theorem}
    

    A corollary of the above theorem is the Riemannian analogue of the four function theorem from \cite{kannan1995isoperimetric}. 
    This version of the theorem is easier to use in our setting and will be the application of the needle decomposition to be 
    used in our setting. 
    \todo[disable]{I think the manifold is supposed to be CD(k,n) below.}
    
    \begin{corollary}[Four Function Theorem, \cite{MR3709716}] \label{cor:four_func}
    	Let $ n \geq 2  $, $ \kappa \in \br $ and $ N \in \left( - \infty, 1 \right) \cup \left[ n , \infty   \right] $. 
    	Let $ \left( M , d , \mu  \right) $ be a $ n $-dimensional weighted complete Riemannian manifold satisfying the $ CD\left( \kappa , N  \right) $ condition. 
    	For $ i = 1,2,3,4 $, consider  integrable functions $ f_i :  M \to \br^{ \geq 0 }$ satisfying  $ \int_M \abs{f_i(x)} d( x , x_0  )  d\mu   < \infty  $ for some $ x_0 $. 
    	Assume that there are constants $ \alpha , \beta  $ such that $ f_1^{\alpha} f_2^{\beta} \leq  f_3^{\alpha} f_4^{\beta}  $ almost everywhere. 
    	If for every $ CD(  \kappa , N) $ needle $ \eta $ (for which the functions are integrable), the following is satisfied 
    	\begin{equation*}
    	\left(\int f_1 d \eta \right)^{\alpha} \left(\int f_2 d \eta \right)^{\beta} \leq \left(\int f_3 d \eta \right)^{\alpha} \left(\int f_4 d \eta \right)^{\beta},
    	\end{equation*} 
    	then, 
    	\begin{equation*}
    	\left(\int f_1 d \mu \right)^{\alpha} \left(\int f_2 d \mu \right)^{\beta} \leq \left(\int f_3 d \mu \right)^{\alpha} \left(\int f_4 d \mu \right)^{\beta}.
    	\end{equation*}
    \end{corollary}
    
    In the case of $ N = \infty $ or $ \kappa = 0 $, we can further simplify the above theorem by decomposing the one-dimensional needles into simpler ``affine" needles. 
    This reduces showing integral inequalities on manifolds to showing inequalities on the real line involving a small number of real parameters. 
    
    \todo[disable]{I also don't quite see why this is a corollary. This seems to further assume $n=1$ Is that $N+1$ or $N$ in $ CD( \kappa, N +1  ) $ below?.}
    
    \begin{lemma}[Reduction to $ N = \infty $ or $ \kappa =0 $ affine case, \cite{MR3709716}] \label{cor:affine_needle}
    	Let $ \mu  $ be a $ CD( \kappa, N +1  ) $ needle on $ \br $ for $ \kappa =0 $ or $ N = \infty $.
    	Consider an integrable, continuous function $ f :  \br \to \br  $ satisfying $ \int_{M} f d \mu = 0 $. 
    	Then, there exists a partition $ \Omega  $ of $ \br $ and a measure $ \nu  $ on $ \Omega  $ and a family of measure $ \left\{ \mu_i  \right\}_{i \in \Omega} $ such that 
    	\begin{enumerate}
    		\item For any measurable set $ A $, 
    		\begin{equation*}
    		\mu\left( A\right) = \int_{\Omega}  \mu_i \left( A\right) d \nu(i).
    		\end{equation*}
    		\item For almost every $ i \in \Omega $, either $ i $ is a singleton or $ \mu_i $ is a $ CD( \kappa, N+1) $ affine needle. 
    		\item For almost every $ i $, 
    		\begin{equation*}
    		\int_i f d \mu_i = 0 .
    		\end{equation*} 
    	\end{enumerate}
    \end{lemma}
    The four function theorem was shown in \cite{kannan1995isoperimetric} for the Euclidean space; for that case we can reduce from general needles to the case of exponential needles.   
    It is conjectured in \cite{MR3709716} that the above theorem should generalize to the $ \kappa \geq 0 $ case for positive $ N $. 
    
    \section{Isoperimetric Inequality on Manifolds with Positive Curvature Lower Bounds} \label{sec:Iso_ineq}
    To show that the random walk mixes fast, we need to show that the conductance of the walk is high. 
    A standard way of showing conductance bounds is through isoperimetry. 
    Bounds on the isoperimetry in convex sets in the Euclidean case were studied by \cite{kannan1995isoperimetric} and there has been a succession of works in the area.     
    The main idea is that since the walks we are interested in are geometric with transition kernel related to the volumes in the convex body, one expects the conductance of the Markov chain induced by the walk to be related to the isoperimetry profile of the set. 
    
    \todo[disable]{Rewrite the following paragraph according to our discussion}
    In the following theorem, we show a version of the isoperimetry theorem from \cite{kannan1995isoperimetric} in the Riemannian setting. 
    Bounds for the Cheeger constant of convex domains of manifolds with Ricci curvature lower bounds were shown in \cite{milman2011isoperimetric}. 
    But in order to apply the isoperimetric theorem to the geometric random walk of interest, we need to get a robust version of the above, similar to the statement in \cite{kannan1995isoperimetric}. 
    We use the localization theorem from \cite{MR3709716} to get the required statement, similar to the proof of the isoperimetric theorem given in the same paper. 
    In this case, we can reduce the required inequalities largely to the Euclidean log-concave case, since the $CD\left(\kappa,N\right)$ condition in this case leads to the log-concavity of the one-dimensional measures. 
    Using the localization technique, we can reduce the required inequalities largely to the Euclidean log-concave case, since the $CD\left(\kappa,N\right)$ condition in this case leads to the log-concavity of the one-dimensional measures. 
    
    \begin{theorem} \label{thm:isoperimetry}
    	Let $ M $ be a manifold with non-negative Ricci curvature, i.e. 
    	\begin{equation*}
    	Ric_{M} \geq 0. 
    	\end{equation*}
    	Let $ K \subset M $ be a strongly convex subset. 
    	Then for any pairwise disjoint subsets $ K_1 , K_2 , K_3 \subseteq K$ such that $ K = \cup_i K_i $ and $ d(K_1, K_3) \geq \epsilon $, we have
    	\begin{equation*}
    	\frac{m}{\epsilon \log 2 } \mathrm{Vol}(K)	\mathrm{Vol}(K_2) \geq \mathrm{Vol}(K_1) \mathrm{Vol}(K_3), 
    	\end{equation*}
    	where $ m = \frac{1}{\mathrm{Vol} \left(K\right)  }\int_K d(x,y) d\mathrm{Vol}(y)  $ for some $ x $. 
    \end{theorem}
    \begin{proof}
    	We will use the localization from \autoref{thm:Klartag} to reduce the problem to an inequality about one-dimensional functions.
    	We will assume that the appropriate sets are closed. 
    	For $ i = 1 , 2 ,3 $, let  $ f_i $ be indicators for $ K_i $ and let $ f_4 $ be $ \left( \epsilon \log 2 \right)^{-1}d\left( x , y  \right) $ for some point $ x $ to be picked later.
    	Then, the required equation reduces to 
    	\begin{equation*}
    	\int_M f_1 \left( x \right) d\mathrm{Vol} \int_{M} f_{3} \left( x\right) d\mathrm{Vol} \leq \frac{1}{\epsilon \log 2 } \int_{M} f_2(x) d\mathrm{Vol} \int_{M} f_4 \left( x\right) d\mathrm{Vol}.
    	\end{equation*}
    	To prove this we use \autoref{cor:four_func}. 
    	Towards this consider a $ CD(0 , \infty ) $ needle $ \eta  $ with the associated minimizing geodesic $ \gamma $. 
    	Note that since $ K $ is strongly convex, we can consider $ \gamma $ is completely contained in $ K $
    	We need to show the show the following inequality. 
    	\begin{equation*}
    	\int_{ M } f_1 \left( x \right) d\eta \int_{M} f_{3} \left( x\right) d\eta \leq \frac{1}{\epsilon \log 2 } \int_{M} f_2(x) d\eta  \int_{M} f_4 \left( x\right) d\eta.
    	\end{equation*}  
    	Using the fact that $ \eta  $ is a push forwards measure, we get 
    	\begin{equation*}
    	\int_{ \gamma^{-1} K_1   } e^{ - \psi \left(x\right)  } dx \int_{ \gamma^{-1} K_3   } e^{ - \psi \left(x\right)  } dx \leq \frac{1}{ \epsilon \log 2}   \int_{ \gamma^{-1} K_2   } e^{ - \psi \left(x\right)  } dx  \int_{ \gamma^{-1} K   }  \abs{x - u }  e^{ - \psi \left(x\right)  } dx 
    	\end{equation*}
    	The last integral follows by noting that we can as assume without loss of generality that the point we are interested in lies on the geodesic of interest and using the fact that $ \gamma $ is a minimizing geodesic. 
    	
    	Note that since, $ \psi $ is a convex function, we have that the $e^{-\psi}$ measure is log-concave. Using the reduction from \cite{kannan1995isoperimetric}, we reduce from the case of arbitrary log concave measures to exponential needles. 
    	
    	\begin{lemma}[\cite{kannan1995isoperimetric}] \label{lem:kannan}
    		Let $f _ { 1 } , f _ { 2 } , f _ { 3 } , f _ { 4 }$ be four non-negative continuous functions defined on an interval $ [a,b] $ and $ \alpha, \beta  > 0 $. Then, the following are equivalent. 
    		\begin{enumerate}
    			\item For every log-concave function $ F: \br \to \br  $, 
    			\begin{equation*}
    			\left( \int _ { a } ^ { b } F ( t ) f _ { 1 } ( t ) d t \right) ^ { \alpha } \left( \int _ { a } ^ { b } F ( t ) f _ { 2 } ( t ) d t \right) ^ { \beta } \leq \left( \int _ { a } ^ { b } F ( t ) f _ { 3 } ( t ) d t \right) ^ { \alpha } \left( \int _ { a } ^ { b } F ( t ) f _ { 4 } ( t ) d t \right) ^ { \beta } ;
    			\end{equation*}
    			\item For subinterval $ [a',b'] \subseteq [a,b] $, and every real $ \delta $, 
    			\begin{equation*}
    			\left( \int _ { a ^ { \prime } } ^ { b ^ { \prime } } e ^ { \delta t } f _ { 1 } ( t ) d t \right) ^ { \alpha } \left( \int _ { a ^ { \prime } } ^ { b ^ { \prime } } e ^ { \delta t } f _ { 2 } ( t ) d t \right) ^ { \beta } \leq \left( \int _ { a ^ { \prime } } ^ { b ^ { \prime } } e ^ { \delta t } f _ { 3 } ( t ) d t \right) ^ { \alpha } \left( \int _ { a ^ { \prime } } ^ { b ^ { \prime } } e ^ { \delta t } f _ { 4 } ( t ) d t \right) ^ { \beta }.
    			\end{equation*}
    		\end{enumerate}
    	\end{lemma}
    	
    	Using \hyperref[lem:kannan]{Lemma \ref*{lem:kannan}}, we reduce the required inequality to one about exponential functions.

    	\begin{equation*}
    	\int_{ \gamma^{-1} K_1   }  e^{ \alpha t   } dt \int_{ \gamma^{-1} K_3   } e^{  \alpha t   } dt \leq \frac{1}{ \epsilon \log 2}   \int_{ \gamma^{-1} K_2   }  e^{  \alpha t   } dx  \int_{ \gamma^{-1} K   }  \abs{x - u }  e^{  \alpha t   } dx 
    	\end{equation*}
    	
    	Note that $ \gamma^{-1} K_i   $ is partition of the interval of support $ \gamma $ into measurable sets, $ J_1 $, $ J_2  $ and $ J_3 $. The required inequality then follows from the proof of Theorem 5.2 in \cite{kannan1995isoperimetric}.
    \end{proof}

    \section{Overlap of One-step Distributions} \label{sec:One_Step}
    To reduce the conductance of the random walk to the isoperimetric inequality, we need to show that the two points that are close by have single step distributions that are close by in total variation distance. 
    To this end we use the techniques followed by \cite{lee2018convergence}.
    \todo[disable]{The following two theorem statements need work: What manifolds do the theorem and lemma apply to? What is $R_1$ in the theorem and what is $R$ in the lemma? What are $P_x, P_y$ in the theorem?}
    In the statement below, we define the Frobenius norm of the Riemannian curvature tensor as, 
    \begin{equation*}
    \norm{R_M} = \mathbb{E}_{u,v \sim N\left( 0, g^{-1} (x)  \right) }   \left[ \ip{R\left( u , \gamma' \right)  \gamma'  }  {v}    \right] ,
    \end{equation*}
    where $ \gamma(t) $ is a geodesic on the manifold.
    Note that the Frobenius norm implicitly depends on the point on the manifold.  
    Also, recall that for any point $ x \in M $, $ P_x\left(x\right) $ denotes the probability that the geodesic walk, specified by \hyperref[alg:GeoWalk]{Algorithm \ref*{alg:GeoWalk}}, starting at a point $ x $ stays at the point in one step. 
    The following definition defines a parameter of the manifold upper bounding the curvature. 
    
    \begin{definition}
    	Let $ \left(M,g\right) $ be a Riemannian manifold.  Let $ R_M $ denote the Riemannian curvature tensor. Then, define 
    	\begin{equation*}
    	R = \max_x  \norm{R_M}_{F}
    	\end{equation*}
    	where the maximum is taken over all points. 
    \end{definition}
    
    \begin{theorem}[One Step Overlap] \label{thm:one_step}
    	Let $ (M,g) $ be a complete Riemannian manifold.
    	Let $ x, y \in M $ be points on the manifold with $ P_x(x) , P_y(y)  \leq c_2$. 
    	Then, for $ \delta^2 \leq \frac{1}{100 \sqrt{n} R } $, we have 
    	\begin{equation*}
    	d_{TV} \left( P_x , P_y   \right) \leq c_2 + O\left(  \frac{1}{\delta} \right) d(x,y) + \frac{1}{25}. 
    	\end{equation*}
    \end{theorem}
    \begin{proof}
    	Note that, since there is a non zero probability of remaining at a given point, the one step distribution is not absolutely continuous with respect to the Riemannian volume. 
    	But, if we exclude this point then the measure does indeed become absolutely continuous with respect to the volume.
    	Now, we need to compare the one step distribution of geodesic walk from two different points. 
    	To do this, we use the following lemma from \cite{lee2018convergence}, noting that the geodesic walk is an instantiation of the Hamiltonian walk with the choice of the Hamiltonian mentioned earlier. 
    	Denote by $ \tilde{P}_x $ and $ \tilde{P}_y $, one step distributions of the geodesic walk without the constraint from the convex set. 
    	
    	\begin{lemma}[see Lemma 25 in \cite{lee2018convergence}]
    		For $ \delta^2 \leq \frac{1}{100 \sqrt{n} R } $, 
    		we have 
    		\begin{equation*}
    		d_{TV} \left( \tilde{P}_x , \tilde{P}_y   \right) \leq O \left( \frac{1}{\delta} \right) d(x,y) + \frac{1}{25}. 
    		\end{equation*}
    	\end{lemma}
    	This gives us the required result. 
    \end{proof}

    	\section{Bounds on Conductance of the Geodesic Walk} \label{sec:Conductance_Bound}
		The following theorem is a version of the general framework of reducing the conductance of a geometric random walk to that of the isoperimetry of the underlying space. 
		We know that nearby points have high overlap of one-step distribution, implying that if two points don't have high overlap, then they must be far apart. 
		Using this we reduce the conductance to isoperimetry of the set given by the \autoref{thm:isoperimetry}. 
		
		Since, we are looking at a walk such that at each step there is a probability of step getting rejected, we need to make certain that we are at a point such that local conductance is not too small. This motivates the following definition. 
        
        \begin{definition}
            Let $ K \subseteq M  $ be a strongly convex set. Define the set of points with high local conductance with respect to the geodesic walk with step size $ \delta $ to be 
            \begin{equation*}
                K_{-\delta} := \{ x \in K :  P_{x} \left( x\right) = 0  \}. 
            \end{equation*}
        \end{definition}
        
        The set $K_{-\delta}$ turns out to be nice: 
        
        \begin{lemma}[Theorem 1.9 in \cite{MR0309010}] \label{fact:Conv_K_Delta}
            Let $ K \subseteq M $ be a strongly convex set on a manifold with nonnegative sectional curvature and 
            let $ K_{-\delta} $ be defined as above. Then, $ K_{-\delta} $ is a strongly convex subset.
        \end{lemma}
        
        Next we show that the volume of $ K_{-\delta} $ is reasonably large compared to the volume of the original set $ K $. 
        To do this we need to use a finer version of the theorem from \cite{MR3709716}. 
        We need the following definition. 
        
        \begin{definition}[Strain, \cite{MR3709716}]
            Given a $1$-Lipshitz function $ u : M \to \br $ and a point $ y $, we say that $ y $ is a strain point of $ u $ if there exist two points $ x $ and $ z $ with $ d(x,z) = d(z,y) + d(y,x) $ such that 
            \begin{equation*}
                u(x )- u(y) = d(x,y) >0
            \end{equation*}
            and 
            \begin{equation*}
                u(z) - u(y) = d(y,z) >0.
            \end{equation*}
            Define the set of strain points on the manifold to be $ Strain[u] $. 
        \end{definition}
        
        \begin{lemma}[\cite{MR3709716}]
            Consider the relation on $ Strain[u] $ by 
            \begin{equation*}
                x \sim y \iff \abs{u     \left(  x \right)    - u(y)   } = d(x,y)  . 
            \end{equation*}
            Then, the relation is an equivalence relation on $ Strain[u] $ and each equivalence class is an image of a minimizing geodesic.  
        \end{lemma}
    
        Denote by $ T^{\circ}\left[u\right] $ the set of equivalence classes. 
        Each element of $ T^{\circ}\left[u\right] $ is associated with the image of a minimizing geodesic. 
    	Using the above definitions, we state the following more general version of \autoref{thm:Klartag}. 
    	It can also be shown that the following theorem implies \autoref{thm:Klartag} using the Monge-Kantorovich duality.
    	\todo[disable]{The set T below hasn't been defined. What is $\nu$?}

        \begin{theorem}[\cite{MR3709716}] \label{thm:Klartag_Lips}
            Let $ M $ be a manifold satisfying the curvature dimension condition $ CD(0,N) $. Let $ u : M \to \br  $ be a 1-Lipschitz function. Then, there exists a measure $ \nu $ on the set $ T^{\circ}\left[u\right] $ and  a family $ \left\{  \mu_I \right\} $ such that 
            \begin{enumerate}
                \item For any measurable set $ A \subseteq M $, the map $ I \to \mu_{I} \left( A\right) $ is well defined $ \nu  $ almost everywhere. 
                For a measurable subset $ S \subseteq T^{\circ}\left[ u  \right]  $, we have $ \pi^{-1} \left( S \right) \subseteq Strain\left[u\right] $ is a measurable subset of $ M $. 
                \item For any measurable set $ A \subseteq M $, we have 
                \begin{equation*}
                    \mu\left(  A \cap Strain\left[u\right]   \right) = \int_{T^{\circ}\left[ u  \right]} \mu_{I} \left(A\right)   d \nu(I)  .  
                \end{equation*}
                \item For $\nu $-almost every $I  \in T^{\circ}\left[ u  \right] $, $ \mu_{I} $ is a $CD(k,N)$-needle supported on $I \subseteq M$. Furthermore, the set $A \subseteq \br $ and the minimizing geodesic $\gamma : A \to M$ may be selected so that $I = \gamma(A)$ and
                \begin{equation*}
                    u\left( \gamma(t)  \right) = t. 
                \end{equation*}
            \end{enumerate}
        \end{theorem}
        
        
        \begin{theorem} \label{thm:rev_iso}
        	Let $ K \subseteq M $ be a strongly convex subset of $ M $ containing a ball of radius $ r $. Then, for $ \epsilon \leq rn^{-1} $, we have 
        	\begin{equation*}
        	\mathrm{Vol}   \left( K - K_{-\epsilon} \right)   \leq \frac{ e  n  \epsilon }{r}  \mathrm{Vol}   \left( K  \right).
        	\end{equation*}
        \end{theorem}
        \begin{proof}
        	We prove this using localization.  
        	Using $\chi$ for the indicator function, we can rewrite our desired inequality as
        	\begin{equation*}
        	\int_{M}   \chi_{ K - K_{-\epsilon}  }\left(x  \right) d \mathrm{Vol} \left(x\right)  \leq \frac{ e  n  \epsilon }{r} \int_{M} \chi_K\left(  x  \right) d \mathrm{Vol} \left(x\right) . 
        	\end{equation*}
        	Let $ x^{*} $ be a point such that there is a ball of radius $ r $ centered around $ x^{*} $ contained within $ K $.
        	We use \autoref{thm:Klartag_Lips} with respect to the function $ u(x) = d(x, x^{*}) $.
        	Note that the strain sets of $u$ are geodesics through $ x^{*} $ and since $K$ is strongly convex and the manifold is geodesically complete, \todo[disable]{Why is it geodesically complete? I guess we need to say something more about the manifold in the theorem statement.} we can extend these geodesics to the boundary. 
        	Using the localization for $ CD\left(0 , n \right) $ manifolds, where $ n $ is the dimension of the manifold, we decompose the measure into needles. 
        	Thus, the required inequality reduces to showing \todo[disable]{Check the integrand}
        	\begin{equation*}
        	\int_{I}   \chi_{K - K_{-\epsilon}}(x) d \eta_I  \leq \frac{ e   n  \epsilon }{r} \int_{I} \chi_{K}(x) d\eta_I .
        	\end{equation*}
        	for every $ CD \left( 0 , n  \right) $ needle $ \eta_I $. 	
        	In the special setting of positive curvature, we need not consider general needles but can restrict ourselves to affine needles:
        	
        	\begin{lemma}[Reduction to $ \kappa = 0 $ affine case, \cite{MR3709716}, \cite{MR2060645}] \label{cor:affine_needle_0}
        		Let $ \mu  $ be a $ CD( 0, N) $ needle on $ \br $.
        		Consider an integrable, continuous function $ f :  \br \to \br  $ satisfying $ \int_{M} f d \mu = 0 $. 
        		Then, there exists a partition $ \Omega  $ of $ \br $ and a measure $ \nu  $ on $ \Omega  $ and a family of measures 
        		$ \left\{ \mu_i  \right\}_{i \in \Omega} $ such that 
        		\begin{enumerate}
        			\item For any measurable set $ A $, 
        			\begin{equation*}
        			\mu\left( A\right) = \int_{\Omega}  \mu_i \left( A\right) d \nu(i).
        			\end{equation*}
        			\item For almost every $ i \in \Omega $, either $ i $ is a singleton or $ \mu_i $ is a $ CD( 0, N) $ affine needle. 
        			\item For almost every $ i $, 
        			\begin{equation*}
        			\int_i f d \mu_i = 0.
        			\end{equation*} 
        		\end{enumerate}
        	\end{lemma}
        	
        	From the above lemma, we can take the needle to be an affine needle. 
        	The condition $CD\left(0,n\right)  $ gives us that $ \eta_j  $ is the pushforward of $ e^{-\phi} $, where  
        	\begin{equation*}
        	\phi'' = \frac{\left(\phi'\right)^2}{n-1}. 
        	\end{equation*}
        	It is easy to check that the solutions to this equations are of the form
        	\begin{equation*}
        	\phi\left(x\right) = a_1 - \left(N-1\right) \log\left(a_2 + \frac{x}{n-1}  \right)
        	\end{equation*}
        	for some constants $ a_1 $ and $ a_2 $.
        	
        	Hence the required needles have density 
        	\begin{equation*}
        	e^{-a_1} \left( a_2 + \frac{x}{n-1}    \right)^{n-1}.
        	\end{equation*}
        	We can now reduce to the following claim about one-dimensional integrals. 
        	\begin{lemma}\label{lem:numerical}
        		Let $ \left[ a , b  \right]  \subset \br $ be  a interval. Then, for every affine function $  c_1 x + c_2 $ such that $ c_1x + c_2 $ is positive on $ \left[ a, b \right] $ and $\epsilon \leq \left(b-a\right)n^{-1} $ , we have 
        		\begin{equation*}
        		\int_{a}^{b}  \left( c_1 x + c_2 \right)^{n-1} \, dx \geq \frac{ b-a}{\epsilon ne}   \int_{b}^{b+ \epsilon}   \left( c_1 x + c_2 \right)^{n-1} dx
        		\end{equation*}
        	\end{lemma}
        	\begin{proof}
        		Using the required change of variables $ y = c_1 x + c_2 $, we reduce to the following case 
        		\begin{equation*}
        		\int_{c_1a + c_2}^{c_1b + c_2}  y^{n-1} \, dx \geq \frac{ b-a}{\epsilon n}   \int_{c_1b +c_2}^{  c_1b + c_2 +  c_1\epsilon}   y^{n-1} dx,
        		\end{equation*}
        		where all the limits of integration are non-negative.
                Setting $ c_1a + c_2 = a '  $, $ c_1b + c_2 = b' $ and $ \epsilon' = c_1 \epsilon $, we get 
                \begin{equation*}
                    \int_{a'}^{b'}  y^{n-1} \, dx \geq \frac{ b-a}{\epsilon n}   \int_{b'}^{ b'+ \epsilon'}   y^{n-1} dx.
                \end{equation*}
                
                Expanding the integrals and taking the ratio, we get the following ratio to upper bound. 
                \begin{align*}
                	\frac{\left( b' + \epsilon'  \right )^n - b'^n   }{b'^n - a'^n} =& \frac{b'^n \left(  1 + \frac{\epsilon'}{b'}   \right)^n - 1   }{b'^n - a'^n}  \\
                	=& \frac{ \left(  1 + \frac{\epsilon'}{b'}   \right)^n - 1   }{1- \frac{a'^n}{b'^n} }.  \\
                	\intertext{Note that the last expression is polynomial in $ \epsilon' $ with zero constant term. }
                	\frac{\left( b' + \epsilon'  \right )^n - b'^n   }{b'^n - a'^n}  =& \epsilon' \sum_{i = 1}^{n}  q_i \epsilon'^{i-1}
                	\intertext{The second factor on the right hand side is monotone in $ \epsilon' $. Setting $ \epsilon' = \left( b' - a' \right) n^{-1} $} 
                	 \frac{\left( b' + \epsilon'  \right )^n - b'^n   }{b'^n - a'^n}  \leq & \epsilon' \sum_{i = 1}^{n}  q_i \left(\frac{b' - a'}{n} \right)^{i-1}.
                	 \intertext{Divinding and multiplying by $ \left(b'-a'  \right) n^{-1} $ }
                	 \frac{\left( b' + \epsilon'  \right )^n - b'^n   }{b'^n - a'^n}  \leq  & \frac{\epsilon' n }{b'-a'} \sum_{i = 1}^{n}  q_i \left(\frac{b' - a'}{n} \right)^{i}.
                	 \intertext{Substituting this back into the earlier expression gives}
                	 \frac{\left( b' + \epsilon'  \right )^n - b'^n   }{b'^n - a'^n}  \leq  & \frac{\epsilon' n }{b'-a'} \frac{ \left(  1 + \frac{b'-a'}{nb'}   \right)^n - 1   }{1- \frac{a'^n}{b'^n} }  
                	 \intertext{Setting $ \alpha = b' a'^{-1}  $ }
                	 \frac{\left( b' + \epsilon'  \right )^n - b'^n   }{b'^n - a'^n}  \leq & \frac{\epsilon' n }{b'-a'} \frac{ \left(  1 + \frac{1 - \alpha}{n}   \right)^n - 1   }{1- \alpha^n }. 
                	 \intertext{Factorizing the numerator and denominator, we get }
                	 \frac{\left( b' + \epsilon'  \right )^n - b'^n   }{b'^n - a'^n}  \leq & \frac{\epsilon' }{b'-a'} \frac{ \sum_{i = 0}^{n- 1}  \left( 1 + \frac{1 - \alpha}{n}   \right)^{i}  }{ \sum_{i = 0}^{n-1} \alpha^i  }. 
                	 \intertext{Note that the denominator is greater than one.}
                	 \frac{\left( b' + \epsilon'  \right )^n - b'^n   }{b'^n - a'^n} \leq & \frac{\epsilon' }{b'-a'}  \sum_{i = 0}^{n- 1}  \left( 1 + \frac{1 - \alpha}{n}   \right)^{i}.  
                	 \intertext{Note that $ 1- \alpha \leq 1 $,} 
                	 \frac{\left( b' + \epsilon'  \right )^n - b'^n   }{b'^n - a'^n} \leq & \frac{\epsilon' }{b'-a'}  \sum_{i = 0}^{n- 1}  \left( 1 + \frac{1 }{n}   \right)^{i} \\
                	 \leq & \frac{\epsilon' }{b'-a'}  \sum_{i = 0}^{n- 1}  \left( 1 + \frac{1 }{n}   \right)^{n} \\
                	 \leq &  \frac{\epsilon' }{b'-a'}  \sum_{i = 0}^{n- 1}  e \\
                	 \leq & \frac{\epsilon'n e}{b'-a'}, 
                \end{align*}
                completing the proof of \hyperref[lem:numerical]{Lemma \ref*{lem:numerical}}.
                \end{proof}
            This completes the proof of \hyperref[thm:rev_iso]{Theorem \ref*{thm:rev_iso}}. 
            \end{proof}            

        We next show the bounds on the conductance of our random walk for manifolds of positive curvature.

         	\begin{theorem}[Conductance in Positive Curvature] \label{thm:positive}
             Let $ \left(M,g\right) $ be a Riemannian manifold and $ K\subseteq M $ be a strongly convex set satisfying the following conditions. 
             \begin{itemize}
                 \item $ M $ has positive sectional curvature i.e. $ Sec_{M} \geq 0 $.
                 \item The Riemannian curvature tensor is upper bounded i.e. $ \max \norm{R}_{F}  \leq R $. 
                 \item $ K $ contains a ball of radius $ r $. 
             \end{itemize}
             Then, for any $ 0 \leq s \leq 0.5 $, the geodesic walk with step length  $ \delta^2 \leq \frac{1}{100 \sqrt{n} R }  $ and $ \delta \leq \frac{sr}{4n\sqrt{n}} $  has $s$-conductance
             \begin{equation*}
                \phi_s \geq \Omega\left( \frac{\delta  }{m}\right),
             \end{equation*}
             where  $m =  \frac{1}{\mathrm{Vol} \left(K\right)  }\int_K d(x^*,y) d\mathrm{Vol}(y)$ for some point $ x^{*} $.  
         \end{theorem}
         \begin{proof}
             By \autoref{thm:rev_iso} $, \mathrm{Vol} \left( K_{-\delta\sqrt{n}}   \right)    \geq \left( 1 - \frac{\left(e-1\right)  \delta n \sqrt{n}  }{r} \right)  \mathrm{Vol}(K) $.
             We are considering step length $ \delta \sqrt{n} $ because the length of the Gaussian random variable in $ n $ dimensions is $ \sqrt{n} $ with high probability. 
             Consider an arbitrary measurable subset $ S_1 $ and its complement $ S_2 $. Bounding the conductance of the random walk boils down to bounding the probability of transitioning from $ S_1 $ to $ S_2 $. \todo[disable]{Last sentence not clear.}
             In order to do this, we shall look at points in the sets that have a low chance of crossing over to the other side. 
             Since we are bounding the $ s $-conductance, we consider $ \mathrm{Vol}\left(S_1\right) , \mathrm{Vol}\left(S_2\right) \geq s \mathrm{Vol}\left(K\right)$.
             Since the walks we consider are geometrical in nature, this corresponds to partitioning the set into well separated subsets. 
             With that in mind, define 
             \begin{align*}
             S'_1 = & \left\{ x \in S_1 \cap K_{ - \delta} : P_x \left(  S_2\right) \leq  c_1 \right\} \\
             S'_2 = & \left\{ x \in S_2 \cap K_{ - \delta} : P_x \left(  S_1\right) \leq   c_1 \right\} . 
             \end{align*}
             We can assume without loss of generality that $ \mathrm{Vol}\left(  S'_1 \right)  \geq 0.5 \mathrm{Vol}(S_1) $ and $ \mathrm{Vol}(S'_2 ) \geq 0.5 \mathrm{Vol}(S_2) $. If it were not, it is easy to see that the conductance of these sets is high.
             Consider any points $ x \in S'_1 $ and $ y \in S'_2 $. Then, 
             \begin{align*}
             d_{TV} \left( P_x , P_y   \right) \geq &  1 - P_x\left( S_2  \right) - P_y\left( S_1 \right)  \\
             \geq & 1 - 2c_1.
             \end{align*}
             Since $ x ,y \in K_{- \delta} $, the local conductance $ P_x(x) $, $ P_y(y)  =0  $, applying \autoref{thm:one_step}, we get 
             \begin{equation*}
             O\left( \delta^{-1} \right)	d\left( x , y\right) \geq  1  - 2c_1 -0.04. 
             \end{equation*} 
              Since this is true for an arbitrary pair of points $ x , y $, we get the same bound for $ d(S_1,S_2) $. 
             Denote by $ S'_3 = K_{-\delta} \cap \left(  S'_1 \cup  S'_2 \right)^{c}   $. 
             We now apply \autoref{thm:isoperimetry} to the partition of $ K_{-\delta} $ given by the sets $ S'_1 $, $ S'_2 $ and $ S'_3 $. 
             \begin{equation*}
             \mathrm{Vol}\left(  S'_3 \right) \geq \Omega \left( \delta  \right)  \frac{ 1   }{ m \mathrm{Vol}\left( K_{-\delta} \right)}\mathrm{Vol}\left( S'_2 \right) \mathrm{Vol}\left( S'_1 \right)
             \end{equation*}
             This gives us, 
             \begin{align*}
             \int_{S_1}   P_x\left( S_2  \right) dx \geq &  \frac{c_1}{2}  \frac{\mathrm{Vol}(S'_3)}{\mathrm{Vol}(K)}  \\
             \geq &  \Omega\left( \delta  \right)  \frac{ \mathrm{Vol}\left( S'_2 \right) \mathrm{Vol}\left( S'_1 \right)   }{ m  \mathrm{Vol}(K) \mathrm{Vol}\left( K_{-\delta}  \right)} \\
             \geq & \Omega\left( \delta  \right) \frac{ \min\{\mathrm{Vol}\left( S'_2 \right), \mathrm{Vol}\left( S'_1 \right) \}  }{ m \mathrm{Vol}(K)}   .
             \end{align*}
             Here, $ m = \int_K d( x^{*}  , x )  d\mathrm{Vol}(x)  $ for some point $ x^* $. 
             Note that $ m \leq D $ where $ D $ is the diameter of the set. 
             This show the required bound on the conductance that we require. 
         \end{proof}

 		The theorem below shows that Markov chains with high conductance have small mixing times as required. Before that we state a lemma.

        \begin{lemma}[$s$-Conductance Implies Fast Mixing , see \cite{vempala2005geometric}] \label{lem:cond}
              			Let $ \mu_t $ be the distribution of the random walk after $ t $ steps starting at $ X_0 $, distributed as $ \mu_0 $. 
                         Let $ H_s = \sup  \left\{ \abs{\mu_0 \left(A\right) - \mu(A) } :  \mu\left(A\right) \leq s   \right\} $
                         where $ \mu $ is the stationary distribution of the random walk. 
                          Then, 
             			\begin{equation*}
             				d_{TV} \left( \mu_t , \mu    \right) \leq H_s +     \frac{H_s}{s} \left( 1 - \frac{\phi_s^2}{2}  \right)^{t}.
             			\end{equation*}
             			 
        \end{lemma}
     
 		
 		\sampling*
 		\begin{proof}
             First note that $ H_s \leq Hs $. 
 		      By \autoref{thm:positive}, we have that for $ \delta^2 = O \left(   \frac{1}{\sqrt{n}  R }  \right)  $ and $ \delta \leq \frac{sr}{4n \sqrt{n}} $  , we have  $ \phi^2 = \Omega\left(  \frac{s^2r^2}{m^2 n^3  \max\left\{1,R\right\}  }  \right) $ and thus from \autoref{lem:cond}, we have for each $ s = \epsilon_1 H^{-1} $,
                \begin{equation*}
                    d_{TV} \left( \mu_t , \mu    \right) \leq \epsilon_1 + H \left( 1 - \Omega  \left(  \frac{s^2r^2}{m^2 n^3 \max\left\{1,R\right\} } \right)   \right)^t. 
                \end{equation*}
                Picking $ \epsilon_1 = \epsilon /2 $ and $ t = O \left( \frac{H^2 m^2 n^3 \max\left\{1,R\right\}}{r^2\epsilon^2} \log \left( H / \epsilon     \right) \right) $, we get the desired result. 
 		\end{proof}
 		
         
         If we assume a strictly positive lower bound on the Ricci curvature, then by the Bonnet--Myers Theorem we get an upper bound on the diameter of the manifold and thus on the diameter of the convex set. 
         
         \begin{fact}[Bonnet--Myers Theorem, see \cite{MR2229062}]
             Let $ M $ be a complete Riemannian manifold, of dimension $ n \geq 2 $, such that there exists a constant $ \kappa > 0 $ for which 
             \begin{equation*}
                Ric_{M} \geq \left( n-1 \right) \kappa.
             \end{equation*} 
             Then, $ M $ is compact with diameter lesser than $  \pi/ \sqrt{\kappa} $. 
         \end{fact}
         
         This gives us a diameter-independent bound on mixing time. 
         This can also be achieved by considering diameter independent isoperimetric inequalities under the strict positivity assumption.  
         
         With slightly different assumption on the convex subset, we can show bounds on the conductance rather than the $ s $-conductance. The proof of the theorem is the same as that of the theorem above but since each point has high conductance, we do not need to restrict to the large sets. 
         
         \begin{theorem}[Conductance in Positive Curvature with assumptions on local conductance] \label{thm:positive_2}
         	Let $ \left(M,g\right) $ be a Riemannian manifold and $ K\subseteq M $ be a strongly convex set satisfying the following conditions. 
         	\begin{itemize}
         		\item $ M $ has non-negative sectional curvature i.e. $ Ric_{M} \geq 0 $.
         		\item The Riemannian curvature tensor is upper bounded i.e. $ \max \norm{R}_{F}  \leq R $. 
         		\item $K$ has diameter $ D $. 
         		\item Each point in $ K $ has local conductance greater than $ c_2 $. 
         	\end{itemize}
         	Then, the geodesic walk with step length  $ \delta^2 \leq \frac{1}{100 \sqrt{n} R }  $ has 
         	\begin{equation*}
         	\phi \geq \Omega\left( \frac{\delta \left(1-c_2- 0.04 \right)^2  }{m}\right). 
         	\end{equation*}
         	where  $ m = \frac{1}{\mathrm{Vol}\left(K\right)} \int_K d( x^{*}  , x )  d\mathrm{Vol}(x)  $ for some point $ x^* $. 
         \end{theorem}
         \begin{proof}
			The proof is similar to that of \autoref{thm:positive}.
			In order to bound the conductance, we shall look at points in the sets that have a low chance of crossing over to the other side. 
			Since the walks we consider are geometrical in nature, this corresponds to partitioning the set into well separated subsets. 
			With that in mind, define 
			\begin{align*}
			S'_1 = & \left\{ x \in S_1  : P_x \left(  S_2\right) \leq  c_1 \right\}, \\
			S'_2 = & \left\{ x \in S_2 : P_x \left(  S_1\right) \leq   c_1 \right\} . 
			\end{align*}
			We can assume without loss of generality that $ \mathrm{Vol}\left(  S'_1 \right)  \geq 0.5 \mathrm{Vol}(S_1) $ and $ \mathrm{Vol}(S'_2 ) \geq 0.5 \mathrm{Vol}(S_2) $. If it were not, it is easy to see that the conductance of these sets is high.
			Consider any points $ x \in S'_1 $ and $ y \in S'_2 $. Then, 
			\begin{align*}
			d_{TV} \left( P_x , P_y   \right) \geq &  1 - P_x\left( S_2  \right) - P_y\left( S_1 \right)  \\
			\geq & 1 - 2c_1.
			\end{align*}
			Since, from the assumption, the local conductance $ P_x(x) $, $ P_y(y)  \leq c_2  $, applying \autoref{thm:one_step}, we get 
			\begin{equation*}
			O\left( \delta^{-1} \right)	d\left( x , y\right) \geq  1 - c_2 - 2c_1 -0.04. 
			\end{equation*} 
			Since this is true for an arbitrary pair of points $ x , y $, we get the same bound for $ d(S_1,S_2) $. 
			Let $ S'_3 = K\cap \left(  S'_1 \cup  S'_2 \right)^{c}   $. 
			We now apply \autoref{thm:isoperimetry} to the partition of $ K $ given by the sets $ S'_1 $, $ S'_2 $ and $ S'_3 $ 
			with $ c_1 = \left(1- c_2 -0.04\right)/4 $:
			\begin{equation*}
			\mathrm{Vol}\left(  S'_3 \right) \geq \Omega \left( \delta  \right)  \frac{ 1   }{ m \mathrm{Vol}\left( K \right)}\mathrm{Vol}\left( S'_2 \right) \mathrm{Vol}\left( S'_1 \right).
			\end{equation*}
			This gives us, 
			\begin{align*}
			\int_{S_1}   P_x\left( S_2  \right) dx \geq &  \frac{c_1}{2}  \frac{\mathrm{Vol}(S'_3) }{\mathrm{Vol}(K)} \\
			\geq &  \Omega\left( \left(1-c_2- 0.04 \right)^2 \delta  \right)  \frac{ \mathrm{Vol}\left( S'_2 \right) \mathrm{Vol}\left( S'_1 \right)   }{ m \mathrm{Vol}\left( K \right) ^2} \\
			\geq & \Omega\left(  \left(1-c_2- 0.04 \right)^2   \delta  \right) \frac{ \min\{\mathrm{Vol}\left( S'_2 \right), \mathrm{Vol}\left( S'_1 \right) \}  }{ m \mathrm{Vol}(K) }   .
			\end{align*} 
			Note that $ m \leq D $ where $ D $ is the diameter of the set. 
			This show the required bound on the conductance that we require. 
         \end{proof}

         	\begin{lemma}[Conductance Implies Fast Mixing , see \cite{vempala2005geometric}] \label{lem:cond2}
         	Let $ \mu_t $ be the distribution of the random walk after $ t $ steps starting at $ X_0 $, distributed as $ \mu_0 $, with $ H  = \norm{ \mu_0 / \mu  } $. Then, for each $ \epsilon > 0 $, 
         	\begin{equation*}
         	d_{TV} \left( \mu_t , \mu    \right) \leq  \epsilon + \sqrt{\frac{H}{\epsilon}} \left( 1 - \frac{\phi^2}{2}  \right)^t .
         	\end{equation*}
         \end{lemma}
     
     Putting together the previous theorem and lemma we get: 
         
         \begin{theorem}[Mixing Time Bound with assumptions on local conductance] \label{thm:sampling2}
         	Let $ \left( M , g  \right) $ be a manifold satisfying the conditions in \autoref{thm:positive_2} and let $ K \subset M $ be a strongly convex subset.  
         	Let the starting distribution be an $ H $-warm start, then for $ \delta^2 \leq \frac{1}{100 \sqrt{n} R_1 }  $ and 
         	\begin{equation*}
         	t = O \left( \frac{ m^2  }{\delta^2 \left( 1 - c_2 - 0.04  \right)^4} \log \left( \frac{H}{\epsilon}    \right) \right)
         	\end{equation*}
         	steps, the distribution $ \mu_t $ is $ \epsilon $ close to the uniform distribution on convex set in total variation. 
         \end{theorem}

    \section{Reduction from Sampling to Optimization for Positive Curvature} \label{sec:Optimization}
    In this section, we focus on the algorithms for solving convex programs on convex sets on manifolds. 
    That is, given access to a membership oracle to a strongly convex set $ K  \subseteq M$, an evaluation oracle for a convex function $ f : K \to \br $ and an error parameter $ \epsilon $, the algorithm needs to output a point $ y \in K$ such that 
    \begin{equation*}
        f(y)  - \min_x f(x) \leq \epsilon. 
    \end{equation*}
    Towards this goal, we adapt the simulated annealing algorithm from the Euclidean setting to the Riemannian setting. 
    Given a function $ f $ and a ``temperature" $ T $, define the density
    \begin{equation*}
         \pi_{f,T} \sim  e^{ -\frac{f}{T}  }. 
    \end{equation*}
    Intuitively, the function puts more weight on points where the function attains small values and sampling from the distribution is likely to output points near the minima of the function $ f $ for low enough temperature $ T $. 
    The issue is that sampling from the distribution for a low enough temperature is a priori as hard as solving the initial optimization problem. 
    The way to get around this issue to set a temperature schedule in which one progressively ``cools" the distribution so that the sample from the previous temperature acts to make it easier to sample from the next temperature. 
    Once we reach a low enough temperature, the sample we attain will be close to the optimum point with high probability. 
     
    Since we need to sample from a distribution proportional to the $ e^{-\frac{f}{T}}  $, the natural idea would be to use the Metropolis filter with respect to the original uniform sampling algorithm. 
    This naturally leads to the following algorithm for sampling from the required distribution 
    (this is same as Algorithm \ref*{alg:Geowalk_LC1} reproduced here for convenience):
     
     \begin{algorithm}[H]
         \SetAlgoLined
         \caption{Adapted Geodesic Walk on a Manifold $ M $ and function $ f $.}
         \label{alg:Geowalk_LC}
         \DontPrintSemicolon
         \KwIn{Dimension $ n $, Convex set $ K $, Convex function $ f : K \to \br $, Starting Point $ X_0 \in K $, Step Size $ \eta $, Number of Steps $ N $.}
         \For{$ \tau \leq N $}{
            Pick $ u_{\tau+1} \leftarrow N\left( 0 ,I   \right)  $ in $ T_{X_i}  M $ where the covariance is with respect to the metric. \;
             \eIf{$ y = \exp_{X_{\tau}}  \left( \eta\, u_{\tau+1}  \right) \in K $}{
                With probability $ \min \left( 1 , e^{-f( y  ) + f(X_{\tau})  } \right) $, set $ X_{\tau+1} \leftarrow y  $. \;
                With the remaining probability, set $  X_{\tau+1} \leftarrow X_{\tau} $. \;
         }{
            Set $ X_{\tau+1} \leftarrow X_{\tau} $. \;
            }
        }
         \KwOut{Point $ X_N \in K $ sampled approximately proportional to $ e^{-f} $.}
     \end{algorithm}

     We analyze the above algorithm over manifolds with non-negative Ricci curvature. This algorithm serves as a step in simulated 
     annealing.
     The sequence of temperatures $T_0, T_1, \ldots$ used by simulated annealing is called the temperature schedule or the cooling schedule. 
     
     \begin{algorithm}
            \SetAlgoLined
            \caption{Simulated Annealing on Manifold $ M $.}
            \label{alg:Optimization}
            \DontPrintSemicolon
            \KwIn{Dimension $ n $, Convex set $ K $, Convex function $ f : K \to \br $, Starting Point $ X_0 \in K $, Number of Iterations $ N $, Temperature Schedule $ T_i $.} 
            \For{$ \tau \leq N $}{
                Sample $ X_{\tau }  $ according to distribution $ \pi_{f, T_i} $  using \hyperref[alg:Geowalk_LC]{Algorithm \ref*{alg:Geowalk_LC}} for  $ \pi_{f, T_i} $ with starting point $ X_{\tau -1} $. \;
            }
            \KwOut{Point $ X_N \in K$ that approximately minimizes $ f $.}
     \end{algorithm}


    The main thing to specify in the design of the algorithm above is the temperature schedule, that is, the sequence of temperatures $ T_i $ from which we sample. 
    We need to set the schedule so that the distributions $\pi_{f,T_i}$ and $\pi_{f,T_{i+1}}$ are close for each $i$, yet keeping the length of the schedule small.  
    Following \cite{kalai2006simulated}, we use
    \begin{equation*}
        T_{i+1} = \left(1 - \frac{1}{\sqrt{n}}\right) T_i
    \end{equation*}
     which provides us with the required guarantees. 
     
    
    Note that since we are interested in optimizing convex functions, we need to sample from the analogues of log-concave distributions. 
    For showing mixing of the random walk, we follow the same technique as in the uniform case. 
    That is, we show that the walk has high conductance by reducing the conductance of the walk to isoperimetry of the set with respect to the required measure. 
    Towards proving isoperimetry, we adapt \autoref{thm:isoperimetry} to the above setting again using \autoref{thm:Klartag}, but now with respect to a weighted Riemannian manifold. 
    Since the measure of interest has density which is an exponential of a convex function, the question reduces to the positive curvature case. 
    
    \begin{theorem} \label{thm:isoperimetry_lc}
        	Let $ M $ be a manifold with positive curvature and let $ K $ be a strongly convex subset. 
        	Let $ g : K \to \br $ be a geodesically convex function and consider the measure with density proportional to 
        	\begin{equation*}
        		\pi_g \sim e^{-g},
        	\end{equation*}
            with respect to the Riemannian volume form. 
        Then for any disjoint $ K_1 , K_2 , K_3 $ such that $ K = \cup_i K_i $ and $ d(K_1, K_3) \geq \epsilon $, 
        \begin{equation*}
        \frac{m}{c\epsilon} \pi_g(K)	\pi_g(K_2) \geq \pi_g(K_1) \pi_g(K_3), 
        \end{equation*}
        where $ m = \int_K d(x,y) d\pi_g(y)  $ for some $ x $. 
    \end{theorem}
    \begin{proof}
    	We apply \autoref{thm:Klartag} with $ d\mu = e^{-g} d\mathrm{Vol} $. 
    	The generalized Ricci curvature then becomes 
    	\begin{equation*}
    		Ric_{\mu , N} = Ric_{M} + Hess_{g}. 
    	\end{equation*}
    	Since $ g $ is convex, we have $ Hess_{g}\left( v ,v\right) \geq 0 $ for all $ v  \in TM$. 
    	Thus, 
    	\begin{equation*}
    		Ric_{\mu , N} \geq 0 
    	\end{equation*}
    	as in the setting of \autoref{thm:isoperimetry}. 
    \end{proof}
    
    Next, we extend the one step overlap of the distribution to the current case. 
    We assume smoothness on the function in the following lemma. We remark that the technique from \cite{lovasz2007geometry} to smoothen the function by taking local averages is likely to work obviating the need for smoothness assumption though we cannot prove it.
    \todo[disable]{Quantify $c_2, c'_2$ below.}
  
%
    \begin{theorem}
    	  Let $g: M \to \br$ be convex and $L$-Lipschitz, and let $P_x^{g}$ be the one step distribution induced by the walk in \hyperref[alg:Geowalk_LC]{Algorithm \ref*{alg:Geowalk_LC}} starting at point $x \in M$. 
    	Let $ x, y \in M $ be points such that $ P_x \left(x\right) , P_y\left(y\right) \leq c_2  $ for some constant $ c_2 \leq 1 $, then for $\delta^2 = O\left( \frac{1}{\sqrt{n}  \left( 1+R \right) L^2  }\right) $, we have 
    	\begin{equation*}
    		d_{TV} \left(  P_x^{g}  ,  P_y^{g}   \right) \leq c_2 + 0.01 + O\left(  \frac{1}{\delta} \right) d(x,y) + \frac{1}{25}
    	\end{equation*}
    \end{theorem}
    \begin{proof}
        As before, we account for the probability of rejection in a step. 
        The rejection probability because of the convex set is given as before by 
        \begin{equation*}
            P_x\left(x\right) \leq c_2.
        \end{equation*}
        We need to account for the probability of rejection because of the additional Metropolis filter given by the function. 
        The probability of rejection is upper bounded by 
        \begin{equation*}
            1 - e^{- \abs{g(y) - g(y') } }
        \end{equation*}
        where $ y' $ is proposed next point. 
        Note that since $ g $ is Lipschitz, we have 
        \begin{equation*}
            \abs{g\left(y'\right)  - g\left(y\right) } \leq L d\left(y,y'\right).
        \end{equation*}
        Thus, the probability of rejection is upper bounded by 
        \begin{equation*}
            1 - e^{- L d\left(y,y'\right)}. 
        \end{equation*}
        From the definition of the walk, we can take $ d\left( y , y'  \right) \leq O\left(\delta \sqrt{n}\right)$ (the probability that larger vectors are drawn can be absorbed into the total variation distance of the one step distribution without the filter). 
        So the rejection probability is bounded above by
        \begin{equation*}
        1 - e^{- L \delta \sqrt{n}}. 
        \end{equation*}
        We get this to be a constant by setting $ \delta^2 = O\left( \frac{1}{n \left(1+R\right)  L^2  }\right) $ assuming $ L>1 $. 
        Having bounded the probability of rejection, we can use \autoref{thm:one_step} to bound the total variation distance. 
    \end{proof}

    \todo[disable]{In the above, the probability is not quite a constant because of the dependence on $R$. How do we account for that?}
    
    With the above theorems, we get the following bound for the mixing time. 
    
    \begin{theorem}[Mixing Time for Sampling from Gibbs distributions] \label{thm:sampling_logconcave}
    	Let $ M $ be a manifold and $ K  \subseteq M $ be a strongly geodesically convex subset. 
    	Let $ g : K \to \br $ be a geodesically convex function on $ K $ with Lipschitz constant $ L $. 
    	Let $\pi_{g,T} \left(x\right) \sim e^{ - \frac{g}{T}  }$ and let $ m_T =   \int_K  d(x , x^{*})    d\pi_{g,T}(x)  $ for some $x^* \in K$. 
    	Let $ r$ be the radius of the largest ball contained within $ K $. 
    	Let $ \mu_t $ be the distribution of the adapted geodesic walk for $ \pi_{g,T} $ after $ t $ steps, starting from a distribution that is $ H $-warm for $ \pi_{g,T} $. 
    	Then, for 
    	\begin{equation*}
    	k= O \left( \frac{H^2 m^2 n^3 \left(1+R\right) L^2}{r^2 T^2 \epsilon^2} \log \left( H / \epsilon     \right) \right)
    	\end{equation*}
    	steps, $ \mu_k $ is $ \epsilon $-close to $ \pi_{g,T} $ in total variation. 
    \end{theorem}

    
     Next, we show  that sampling from the distribution that we get does indeed give us a point close to the required optimum. 
    That is, for a small enough temperature $ T $, the expected value of the function under the distribution is small. 
    Using Markov's inequality, one can then bound the probability that a sample produces a value much larger than the optimal value. 
    In the Euclidean case, this inequality was shown by \cite{kalai2006simulated} for the case of linear functions and was explicitly shown for general convex functions in \cite{belloni2015escaping}. 
    
    \begin{theorem}   \label{thm:low_temperature_expectation}
        Let $ K \subseteq M  $ be a strongly convex subset of manifold $ M $ with non-negative Ricci curvature. 
        Let $ g : K \to \br $ be a convex function with minimum value zero. Let $ X $ be sampled from $ \pi_{g,T} $. Then, 
        \begin{equation*}
            \mathbb{E}_{ \pi_{g,T} } \left( g(X)   \right) \leq T\left(n+1\right) + \min_{x \in K} f\left( x \right).  
        \end{equation*}
    \end{theorem}
    \begin{proof}
    In order to prove the above theorem, we use a localization technique similar to the one used in \hyperref[thm:rev_iso]{Theorem \ref*{thm:rev_iso}}. 
    We note the above theorem in terms of integrals in the next theorem, to make it amenable to localization. 

    \begin{theorem}
        Let $ K \subseteq M $ be a convex subset of the manifold $ M $ satisfying the curvature dimension condition $ CD(0,N) $. For any convex function $ g : K \to \br $ with $ \min g(x) = 0 $, we have 
        \begin{equation*}
            \int_{K}    e^{-g}  g \leq \left(N+1\right)\int_K e^{-g}.
        \end{equation*}
    \end{theorem}
    \begin{proof}
    	For the sake of contradiction, assume that the inequality is false.
        Since $ g $ is a convex function, it has a unique minimizer $ x^{*} $ in $ K $. 
  		By assumption, we have $ g(x^{*}) = 0 $. 
        Consider the $1$-Lipschitz function $u$ defined by
        \begin{equation*}
        	  u(x) = d\left( x^{*} , x   \right). 
        \end{equation*}
        Note that the strain sets of $ u $ are geodesic segments through the minimizing point $ x^{*} $.
        Applying \autoref{thm:Klartag_Lips} with respect to $ u $, we get a decomposition of measure $ \eta_{I} $. 
        Note that the support of $ \eta_I $ is a geodesic passing through $x^{*}$. 
        From the decomposition, we get that for some needle $ \eta_j $, 
        \begin{equation*}
        	\int_{j} \left( e^{g} g   - \left(N+1\right) e^{g}   \right) d \eta_{j} \geq 0. 
        \end{equation*}
       	Note that $ \eta_j $ is a $CD\left(0,N\right)  $ needle.
        We now note the reduction to the affine case. 
        
        \begin{lemma}[Reduction to $ \kappa = 0 $ affine case, \cite{MR3709716}, \cite{MR2060645}] \label{cor:affine_needle_0}
            Let $ \mu  $ be a $ CD( 0, N) $ needle on $ \br $.
            Consider an integrable, continuous function $ f :  \br \to \br  $ satisfying $ \int_{M} f d \mu = 0 $. 
            Then, there exists a partition $ \Omega  $ of $ \br $ and a measure $ \nu  $ on $ \Omega  $ and a family of measure $ \left\{ \mu_i  \right\}_{i \in \Omega} $ such that 
            \begin{enumerate}
                \item For any measurable set $ A $, 
                \begin{equation*}
                \mu\left( A\right) = \int_{\Omega}  \mu_i \left( A\right) d \nu(i).
                \end{equation*}
                \item For almost every $ i \in \Omega $, either $ i $ is a singleton or $ \mu_i $ is a $ CD( 0, N) $ affine needle. 
                \item For almost every $ i $, 
                \begin{equation*}
                \int_i f d \mu_i = 0
                \end{equation*} 
            \end{enumerate}
        \end{lemma}
           
       	From the above lemma, we can take the needle to be an affine needle. 
       	The condition for $CD\left(0,N\right)  $, gives us that $ \eta_j  $ is the pushforward of $ e^{-\phi} $, where  
       	\begin{equation*}
       		\phi'' = \frac{\left(\phi'\right)^2}{N-1}. 
       	\end{equation*}
       	It is easy to check that the solutions to this equations are of the form
       	\begin{equation*}
       		\phi\left(x\right) = a_1 - \left(N-1\right) \log\left(a_2 + \frac{x}{N-1}  \right)
       	\end{equation*}
       	for some constants $ a_1 $ and $ a_2 $. Thus, we have 
       	\begin{equation*}
       		e^{-\phi(x)} = e^{-a_1} \left( a_2 + \frac{x}{N-1}    \right)^{N-1}. 
       	\end{equation*}
       	Thus, we get 
       	\begin{equation*}
       		\int_{J} \left( e^{g} g   - \left(N+1\right) e^{g}   \right) d \eta_{j} = \int_{\gamma^{-1} j }  e^{-a_1} \left( a_2 + \frac{x}{N-1}    \right)^{N-1} \left( e^ {g \left( \gamma(t) \right)  } g\left( \gamma(t) \right)     - \left(N+1\right) e^{g\left( \gamma(t) \right)  }   \right) d t
       	\end{equation*}
       	With appropriate change of variables, we get the integral to be 
       	\begin{equation*}
       		\int_a^b t^{N-1} \left( e^ {h(t)} h(t) - \left(N+1\right) e^{h(t)}   \right) d t. 
       	\end{equation*}
       	Here $ h $ is the convex function gotten from $ f $ after the affine change of variables.
       	Note that this is an integral over the real numbers, which can be dealt with using elementary tools. 
       	We capture this with the lemma below. 
       	\begin{lemma} \label{lem:Rn_KV}
       		Let $ h : [a,b] \to \br  $ be a convex function with minimum value zero.
       		Then, 
       		\begin{equation*}
       			\int_a^{b}  e^{-h(z)}h(z) z^{n-1} dz \leq \left(n+1\right) \int_a^b e^{-h(z)} z^{n-1} dz.
       		\end{equation*}
       	\end{lemma}
       \begin{proof}
       		Consider the truncated convex cone $ K_1 $  in $ \br^n $ with radii $ a $ and $ b $. Let $ x_1 $ denote the coordinate along the axis of the cone, and let the centres of the circles defining the cone be on the $ x_1 $ axis at $ x_1 = a $ and $ x_1 = b $ respectively. 
       		Consider the convex function $ c(x) = h(x_1) $. 
       		Note that 
       		\begin{equation*}
       			\int_{K_1} e^{-c(x)}c(x) dx = \int_a^{b}  e^{-h(z)}h(z) z^{n-1} dz .
       		\end{equation*}
       		and 
       		\begin{equation*}
       			\int_{K_1} e^{-c(x)}  dx = \int_a^b e^{-h(z)} z^{n-1} dz. 
       		\end{equation*}
       		The required result now follows from the results in \cite{kalai2006simulated} and \cite{belloni2015escaping}. 
       \end{proof}
       From \autoref{lem:Rn_KV}, we get a contradiction. 
       Thus, we have the required bound. 
    \end{proof}

    Now, we use the above theorem with $ N = n $, the dimension of the manifold to get the desired result. 
\end{proof}
	
	Given that we have shown that we can sample from a fixed temperature, we need to show that sampling from a distribution gives a warm start to the next distribution. 
	We do this by adapting the analysis of log-concave distributions to the manifold case using the localization lemma from \cite{MR3709716}. 
	
	\begin{lemma} [adaptation to manifolds of \cite{kalai2006simulated}, \cite{lovasz2003simulated}]
		Let $ K  $ be a strongly convex subset of $ M $ and $ f : K \to \br $ be a convex function. For $ a \geq 0 $, let
		\begin{equation*}
		Z(a) =  \int_{K}   e^{- a f\left(x\right)} dx  .
		\end{equation*}
		Then,
		\begin{equation*}
		Z(a)   Z(b) \leq  Z\left(  \frac{a+b }{2} \right )^2   \left(\frac{\left(a + b\right)^2 }{4ab}\right)^n. 
		\end{equation*} 
	\end{lemma}
	\begin{proof}
		We prove this inequality by using the four function theorem.
		Set $ f_1  =  e^{-a f }  $, $ f_2  =  e^{-b f }  $ and $ f_3 =   e^{- \frac{a+b}{2}f }  $ and $ f_4 =  \left(\frac{\left(a + b\right)^2 }{4ab}\right)^n  e^{- \frac{a+b}{2}f }   $. 
		Note that $ f_1f_2 \leq f_3 f_4  $ from the AM-GM inequality.
		Then, from \autoref{cor:affine_needle_0}, we have that the above theorem is true if the corresponding inequality is true for $ CD\left(0,N\right) $ affine needles. 
		From the above argument, we see that this reduces to showing 
		\begin{equation*}
			\int_{c}^{d} e^{-a h(x)  }  x^{n-1} dx \int_{c}^{d} e^{-b h(x)  } x^{n-1} dx \leq \left(\frac{\left(a + b\right)^2 }{4ab}\right)^n \left(\int_{a}^b e^{- \frac{a+b}{2}h(x) } dx   \right)^2.
		\end{equation*}
		Here $ h $ is the convex function obtained from $ f $ after the affine change of variables.
		The last inequality follows from the inequality in $ \br^n $. 
	\end{proof}
	
	\begin{theorem} \label{thm:adjacent_dist}
         Let $ T_i $ and $ T_{i+1} $ be adjacent temperatures in the algorithm. 
		Then, 
		\begin{equation*}
		\norm{ \pi_{g,T_i} / \pi_{g,T_{i+1}}   } \leq 5.
		\end{equation*}
	\end{theorem}
	\begin{proof}
		Denote by $ \beta_i $ the inverse of $ T_i $. From the definition of the norm, we get 
		\begin{align*}
		\norm{ \pi_{g,T_i} / \pi_{g,T_{i+1}}   } = &  \frac{Z( \beta_{i+1}  )}{ Z(\beta_i) ^2 } \int e^{ - g(x) \beta_i }  e^{  \beta_{i+1}  g(x)  } e^{ - g(x)  \beta_{i} } dx \\
		= &  \frac{Z( \beta_{i+1}  )  Z( 2 \beta_{i} - \beta_{i+1} )     }{ Z(\beta_i) ^2 } . 
		\intertext{We use the fact that $ a^n Z(a) $ is log-concave to get,   }
		\norm{ \pi_{g,T_i} / \pi_{g,T_{i+1}}   } \leq & \left(\frac{ \beta_{i} ^2 }{  \left(2\beta_i - \beta_{i+1} \right)  \beta_{i+1}  } \right)^n . \\
		\intertext{By the choice of the temperature schedule, we get, }
		\norm{ \pi_{g,T_i} / \pi_{g,T_{i+1}}   } \leq & \left( \frac{\left( 1 - 1/ \sqrt{n}  \right)^2   }{2\left( 1 - 1/\sqrt{n}  \right)  -1 }   \right)^n \\
		\leq & \left( 1 +   \frac{1}{n   -  2\sqrt{n} }   \right)^n \\
		\leq & e^{ n/(n -2 \sqrt{n}  ) } \\
		\leq & 5.  
		\end{align*}
	\end{proof}

        Given the above theorem, we write down the running time bound for the simulated annealing algorithm. 
        
    	\simulatedannealing*
        \begin{proof}
            We start with a temperature such that the uniform distribution $ \nu $ on the convex set satisfies $ \norm{ \nu / \mu_0  }  \leq c_5 $ for some constant $ c_5 $. 
            From the choice of temperature schedule, we get that after $ I =  \sqrt{n} \log\left( \frac{T_0 \left(n+1\right) }{\epsilon \delta}   \right)  $ rounds of cooling, we are at the temperature $ \frac{\epsilon \delta }{\left(n+1\right)  } $. 
            Then, from \autoref{thm:low_temperature_expectation}, we get 
            \begin{equation*}
                \mathbb{E} \left[  f\left(z\right)  \right] \leq \min_{x} f(x) + \epsilon \delta . 
            \end{equation*}
            Using Markov inequality, we have that with probability $ 1 - \delta $, we are within $ \epsilon $ of the minimum value.  
            In each phase $ i $, we sample such that the distribution we sample from is at most $ \frac{\delta}{100I} $ in total variation. 
            From \autoref{thm:sampling_logconcave}, this requires 
            \begin{equation*}
            	L_i =  O \left( \frac{H^2 I^2 m_{T_i}^2 n^3 \left(1+R\right) L^2}{r^2 T_i^2 \delta^2} \log \left( \frac{HI}{\delta}    \right) \right)  
            \end{equation*}
            steps starting from a $ H $-warm start. 
            From  \autoref{thm:adjacent_dist}, we can take each adjacent distribution have  $ \norm{ \pi_{g,T_i} / \pi_{g,T_{i+1}}   }  \leq 5 $. 
            Using a standard trick (see for example, Corollary 3.5 in \cite{vempala2005geometric}), we can assume that this is a $ 2000I/\delta $-warm start with probability $ 1 - \delta/2000I $ and with probability $ \delta/2000I $, it is sampled from an arbitrary distribution. 
            This gives us 
            \begin{equation*}
            		L_i =  O \left( \frac{ I^4 m_{T_i}^2 n^3 \left(1+R\right) L^2}{r^2 T_i^2 \delta^4} \log \left( \frac{I}{\delta}  \right) \right)  
            \end{equation*}
            Note that $ m_{T_{i}}  \leq D $, where $ D $ is the diameter of the set $ K $ and $ T_i \geq \frac{\epsilon \delta}{2n} $, which gives us 
            \begin{equation*}
             L_i =  O \left( \frac{D^2 n^7 \left(1+R\right) L^2}{r^2 \epsilon^2 \delta^6} \log \left( \frac{n}{\delta} \log\left( \frac{T_0 \left(n+1\right) }{\epsilon \delta}   \right)        \right)    \log^4\left( \frac{T_0 n }{\epsilon \delta}   \right)    \right)   
            \end{equation*}
            as required. 
        \end{proof}

    
     \newpage 
    
    \bibliographystyle{alpha}
    \bibliography{Manifold_Sampling_Arxiv}
    
    \newpage
    \appendix

\end{document}